\documentclass[12pt]{amsart}
\usepackage{geometry}         
\usepackage{rotating} 

\usepackage{geometry}             

\usepackage{caption}
\usepackage{subcaption}
\usepackage{xcolor}
\usepackage{mathrsfs,amsthm}
\usepackage{amsmath,amssymb}

\usepackage{pinlabel}
\usepackage{fancyhdr}
\usepackage{lastpage}
\usepackage{makecell}
\usepackage{setspace}
\usepackage{longtable}

\usepackage[utf8]{inputenc}

\usepackage{pinlabel,url}
\usepackage{psfrag}
 \usepackage{epsfig}
\usepackage[all]{xy}
\usepackage{tikz}
\usepackage{graphicx}
\usepackage{amsmath,amsthm}
\usepackage{amssymb,amsfonts,latexsym}
\usepackage{epstopdf}
\usepackage[shortlabels]{enumitem}
\usepackage{dsfont}

\usepackage{hyperref}
\usepackage{cleveref}
\hypersetup{
pdfpagemode=FullScreen,
colorlinks=true,
citecolor = blue}

\newtheorem{theorem}{Theorem}[section]
\newtheorem{lemma}[theorem]{Lemma}

\newtheorem{proposition}[theorem]{Proposition}

\newtheorem{corollary}[theorem]{Corollary}

\theoremstyle{definition}
\newtheorem{definition}[theorem]{Definition}

\newtheorem{example}[theorem]{Example}

\newtheorem{remark}[theorem]{Remark}

\theoremstyle{definition}
  \newtheorem*{proof1}{Proof of Theorem \ref{T2}}
  \newtheorem*{proof2}{Proof of Theorem \ref{T3}}

\newcommand{\thistheoremname}{}
\newtheorem{genericthm}[theorem]{\thistheoremname}

\newtheorem*{genericthm*}{\thistheoremname}
\newenvironment{namedthm*}[1]
  {\renewcommand{\thistheoremname}{#1}%
   \begin{genericthm*}}
  {\end{genericthm*}}

\def\HFKhat{\widehat{\mathit{HFK}}}

\newcommand\tsim{\kern-.4em\sim}

\newcommand{\interior}{\mathrm{int}}
\newcommand{\B}{\mathcal{B}}
\newcommand{\N}{\mathcal{N}}

\newcommand{\R}{\mathbb{R}}

\newcommand{\Z}{\mathbb Z}

\newcommand{\sL}{{\mathcal L}}

\newcommand{\Q}{\mathbb Q}

\theoremstyle{empty}

\def\spl{\backslash\backslash}

\begin{document}


\title{Guts in Sutured Decompositions and the Thurston Norm}

\author[Ian Agol]{%
        Ian Agol} 
\address{%
    University of California, Berkeley} 
\email{%
     ianagol@math.berkeley.edu}  
 \author[Yue Zhang]{%
        Yue Zhang} 
\address{%
    University of California, Berkeley
} 
\email{%
     yue\_zhang@berkeley.edu}

\maketitle

\begin{abstract}
	We construct an invariant called guts for second homology classes in irreducible 3-manifolds with toral boundary and non-degenerate Thurston norm. We prove that the guts of second homology classes in each Thurston cone are invariant under a natural condition. We show that the guts of different homology classes are related by sutured decompositions. As an application, an invariant of knot complements is given and is computed in a few interesting cases. Besides, we show that the dimension of a maximal simplex in a Kakimizu Complex is an invariant.

\end{abstract}

\section{Introduction}

Thurston \cite{Thurston} defined a norm  on the homology $H_2(M,\partial M)$ called the {\it Thurston norm} for a compact 3-manifold $M$, which is an important invariant of 3-manifolds. Gabai \cite{Ga1} generalized the definition and defined the notion of sutured manifolds which is related to foliations of manifolds. Furthermore, one can decompose a sutured manifold into two submanifolds, the windows and the guts, specializing the notion of guts defined by Gabai and Kazez for laminations \cite{GK}. Guts can be used to understand how close a manifold is to a fibration. Agol \cite{Ag1} and later Friedl and Kitayama \cite{FK} used guts to show that an irreducible 3-manifold with fundamental group that satisfies a certain group-theoretic property called $RFRS$ is virtually fibered.

Guts is also defined for pared manifolds. It reveals information about the volume of a hyperbolic manifold. Agol, Storm and Thurston \cite{agol2005lower} gave an estimate of the volume  of a manifold from the volume of guts related to incompressible surfaces. Agol \cite{agol2010minimal} and Yoshida \cite{yoshida2013minimal} found the minimal volumes of 2-cusped and 4-cusped orientable hyperbolic 3-manifolds respectively with the help of estimating the volume of guts. 

In this paper, we give a definition for a \emph{facet surface} of a primitive homology class in an irreducible 3-manifold as a maximal collection of disjointly embedded Thurston norm minimizing surfaces for that homology class and define the \emph{guts} of this {homology class} as the nontrivial parts in the sutured decomposition of the 3-manifold along a facet surface. 

In the following, we consider irreducible, orientable manifolds with torus boundary components and non-degenerate Thurston norm.

We first show that the guts for homology classes are well defined.

\begin{theorem} \label{thm:gutsforhomology}

	Let $M$ be an irreducible, orientable $3$-manifold with boundary a disjoint union of tori $\sqcup_{i=1}^n P_i$ and non-degenerate Thurston norm. The guts $\Gamma(M,F)$ for facet surfaces representing a primitive element $z$ in $H_2(M,\partial M;\Z)$ do not depend on the selection of facet surfaces and product decomposition surfaces up to the equivalence of guts.
\end{theorem}

By constructing branched surfaces, we prove that the guts are invariant for second homology classes in each Thurston cone under a natural condition (see Definition \ref{def:guts:restriction}). This is proved by using a nice property of double curve sum operations \cite[Proposition 2.8]{miller2019effect}.

\begin{theorem}\label{thm:samegutsboundarycase}
Let $M$ be an irreducible, orientable $3$-manifold with boundary a disjoint union of tori $\sqcup_{i=1}^n P_i$ and non-degenerate Thurston norm. Let $y,z$ be two elements in an open face of the Thurston sphere. If there is an open segment $(v,w)$ containing $y,z$ in the open face such that the restrictions of $v$ and $w$ on each boundary component are not in opposite orientations, the guts $\Gamma(y)$ is equivalent to $\Gamma(z)$.
\end{theorem}

Agol's criteria for virtual fibering is based on a key finding \cite[Lemma 4.1]{Ag1} and we state a generalization of it. 

\begin{theorem} \label{T2}
	Let $z$ lie on a $k$-codimensional open Thurston cone $\Delta$. We denote the subspace spanned by $\Delta$ as $V$. Then the kernel of the restriction map $\varphi:H_2(M,\partial M;\Z) \to H_2(\Gamma(z),\partial \Gamma(z);\Z)$ is a subspace of $V$.
	
	If there is a neighborhood of $z$ in $\Delta$ which does not contain two elements whose restrictions on a boundary component are in opposite orientations, the kernel of $\varphi$ is exactly $V$ and the rank of the image of $\varphi$ is $k$, i.e. the rank of the image of 
	$$  H^1(M) \to H^1(\Gamma(z))$$
	is $k$.
\end{theorem}

The first part of this theorem is proved by using the canonical foliation on the complement of the guts. We prove the opposite direction via double curve sum operations.

With the help of the preceding theorem, we show that guts of homology classes are related by sutured decompositions.

\begin{theorem} \label{T3}
Let $M$ be an irreducible connected 3-manifold with toral boundary and non-degenerate Thurston norm.	 Let $z$ be an element on a $k$-codimensional $(k>0)$ open Thurston face of the Thurston norm. For any choice of $u$ in any $k'$-codimensional $(k' < k)$ open Thurston face whose closure contains $z$, there exists a $w$ on the open segment $(z,u)$ such that $\Gamma(z)$ can be nontrivially decomposed along a properly norm-minimizing surface representing $w$.
\end{theorem}

 As an application, an invariant of knot complements is given as the guts of a maximal collection of disjoint Seifert surfaces (see Definition \ref{def:gutsofknot}). We find an example of two 2-bridge knots having different guts but the same Floer knot homology in Example \ref{ex:2bridgeknot}.
\newline

{\bf Motivation.} 

Why define the notion of guts associated to a sutured manifold or  a homology class? 
As indicated before, the main theorem is a natural generalization of Thurston's fibered face theorem, 
and a generalization used to understand the virtual fibering  conjecture. But another motivation comes from sutured manifold theory. 
When performing sutured manifold decompositions, the most natural ways one will try are horizontal and vertical decompositions. What the main theorem implies is that after throwing away the product pieces after a maximal such collection of decompositions, the remaining sutured manifold is independent of the choices made. So it gives a weak kind of uniqueness for the initial stages of a sutured manifold hierarchy. Moreover, the sutured Floer homology is invariant under horizontal and vertical decompositions, hence it is natural to look for uniqueness after these operations. 

Another reason to consider these notions is to investigate the Kakimizu complex of isotopy classes of minimal genus surfaces bounding a fixed knot. One consequence of the main theorem is that the dimension of maximal simplices of the Kakimizu complex is an invariant of the complex (and hence of the knot type). A generalization of this statement is proved in Theorem \ref{thm:kakimizu}. Our terminology {\it facet surface} is motivated by these considerations. Moreover, neighboring maximal simplices correspond to replacing one surface with another one intersecting it, and performing an exchange on the order of the guts. So hopefully the results in this paper will give some refined information about the structure of the Kakimizu complex. 

Finally, a motivation for studying these notions comes from trying to find surfaces in the manifold whose complement has maximal volume guts, in order to get lower bounds on the volume of the manifold. It turns out that the structure uncovered in the main theorems is useful for analyzing configurations of surfaces in manifolds of small volume.

{\bf Organization.}

In Section \ref{sec:guts:thurstonnorm} we recall some standard facts about the Thurston norm and sutured manifolds.

In Section \ref{sec:guts:guts}, we define a facet surface of a second homology class in a 3-manifold as a maximal collection of properly norm-minimizing surfaces and the guts of this homology class as the nontrivial parts in the sutured decomposition of the 3-manifold along a facet surface. We show that the guts are invariant among facet surfaces representing the same homology class.

In Section \ref{sec:guts:cone}, we show that if two elements in an open face of the Thurston sphere don't meet any boundary component in opposite orientation, their guts are equivalent. 

After that, we find a relation between the open Thurston cone $\Delta$ and the restriction map $\varphi $ from $H_2(M,\partial M)$ to $H_2(\Gamma(z),\partial \Gamma(z))$ for an element $z$ in $\Delta$ in Section \ref{sec:guts:map}.

With the previous results, we are able to show that guts of second homology classes can be decomposed by surfaces in the manifold in Section \ref{sec:guts:decomposition}.

In Section \ref{sec:guts:example}, we give some examples of guts and define an invariant for knot complements. One can gain intuition of the main theorems from this section before trying to understand the proofs.

In Section \ref{sec:guts:kakimizu}, we prove that the dimension of maximal simplices of a Kakimizu complex is an invariant of the complex.

{\bf Acknowledgements.}
	
This research is supported by a Simons Investigator award and by the Mathematical Sciences Research Institute. We thank Nathan Dunfield for allowing us to use his drawings. The second author is thankful to Chi Cheuk Tsang, Yi Liu, Yi Ni, Maggie Miller, Michael Landry, Andras Juhasz, and Jiajun Wang for helpful comments and conversations. The second author is grateful to the first author for the guidance.

\section{The Thurston norm and Sutured Decompositions} \label{sec:guts:thurstonnorm}

This section reviews some background material on the Thurston norm \cite{Thurston} and sutured manifolds \cite{Ga1}. 

We denote $\eta(X)$ as a regular neighborhood of $X$. For a properly embedded surface $S \subset M$, $\eta(S)$ is an open tubular neighborhood of $S$ homeomorphic to the normal bundle, and we use $M \spl S$ to indicate ${M\backslash \eta(S)}$.

In this paper, we sometimes omit the $\Z$ coefficient in homology and cohomology groups. 

\begin{definition}
	
Let $S$ be a compact connected orientable surface. Define 
$\chi_{-}(S)=\max\{-\chi(S), 0\}$. For a disconnected compact orientable surface
$S$, let $\chi_-(S)$ be the sum of $\chi_-$ evaluated on the connected subsurfaces of $S$.

\end{definition}
In \cite{Thurston}, Thurston defined a pseudonorm on $H_2(M,\partial M;\R)$
and $H_2(M;\R)$, and this was generalized by Gabai \cite{Ga1}.

\begin{definition}
Let $M$ be a compact oriented 3-manifold. Let $K$ be a subsurface
of $\partial M$. Let $z\in H_2(M,K)$. Define the \emph{Thurston norm} of $z$
to be
$$x(z)=\min\{\chi_-(S)\ |\ (S,\partial S)\subset (M,K),\ \text{and} \ [S]=z\in H_2(M,K)\}.$$
\end{definition}
Thurston proved that $x(nz)=n x(z)$, and therefore $x$ may
be extended to a norm $x:H_2(M,K;\Q)\to \Q$. Then $x$
is extended to $x:H_2(M,K;\R)\to \R$ by continuity.  

Denote by $B_x(M)$ the unit ball of the Thurston norm on $H_2(M,\partial M)$. By \cite[Theorem 2]{Thurston}, $B_x(M)$ is a polyhedron and hence we can decompose $B_x(M)$ into faces of various dimensions. We call $B_x(M)$ \emph{the Thurston ball}, $\partial B_x(M)$ \emph{the Thurston sphere}, each face of $\partial B_x(M)$ \emph{a Thurston face} and a cone formed by the origin and a Thurston face \emph{a Thurston cone}. When we say $z$ is on a Thurston face, we actually mean a multiple of $z$ is on a Thurston face.

\begin{remark}
	When the interior of $M$ is a hyperbolic manifold of finite volume, the Thurston norm is non-degenerate.
\end{remark}

\begin{definition}
Let $M$ be a compact $3$-manifold. Let $S$ be an oriented surface representing $\alpha\in H_2(M,\partial M)$ with $\alpha\neq 0$. We say that $S$ is {\emph{norm-minimizing}} if the following conditions hold:

\begin{itemize}
\item $S$ has no null-homologous subset of components in $H_2(M,\partial M)$,
\item $\chi_-(S)=x_M(\alpha)$.

\end{itemize}

\end{definition}

When the boundary of a manifold is not just a union of tori, we define a notion of taut surfaces.

\begin{definition}
Let $M$ be a compact $3$-manifold. An oriented surface $(S,\partial S)$ is \emph{taut} if 
	\begin{itemize}
		\item $S$ is incompressible and no collection of components is homologously trivial in $H_2(M,\partial M)$ and,
		\item $S$ has minimal $\chi_-$ of all surfaces representing $[S,\partial S]$ in $H_2(M, \eta(\partial S))$, we have $x([S]) = \chi_-(S)$. Here $\eta(\partial S)$ means the tubular neighborhood of $\partial S$.

	\end{itemize}
 \end{definition}

\begin{lemma} \label{tnms}
	If $P$ is a norm-minimizing (or taut) surface, and $Q$ is a union of components of $P$, then $Q$ is also a norm-minimizing (or taut) surface.
\end{lemma}

This follows immediately from the triangle inequality for the Thurston norm.

In this chapter, we use an important tool called double curve sum to analyze the guts of different homology classes. 

\begin{definition}
	Let $S$ and $T$ be two properly embedded oriented surfaces which are in general position. By the standard cut-and-paste technique applied to the intersection curves of $S$ and $T$, we can turn the immersed surface $S\cup T$ into a properly embedded oriented surface $S \oplus T$. The surface $S \oplus T$ is called the \emph{double curve sum} of $S$ and $T$.
\end{definition}

The following lemma and corollary show some simple but important facts about double curve sum.

\begin{lemma} \label{lem:dcs}
	If every component of $S \cap T$ is essential in $S$ and $T$, then $\chi_-(S) + \chi_-(T) = \chi_-(S\oplus T )$.
\end{lemma}
\begin{proof}
	This is because $\chi(S) + \chi(T) = \chi(S\oplus T )$ and $S\oplus T$ does not have components of positive Euler characteristic.
\end{proof}

\begin{lemma} \label{lem:dcs2}
	Let $S_1$ and $S_2$ be two properly embedded oriented surfaces in general position such that every component of $S_1\cap S_2$ is essential in $S_1$ and $S_2$. If $S_1\oplus S_2$ is a norm-minimizing surface, then $S_1$ and $S_2$ have the minimal $\chi_-$ among all surfaces representing the same element in the second relative homology group $H_2(M,\partial M)$ respectively.
	\end{lemma}
\begin{proof}
	This is because of the equality $\chi_-(S_1) + \chi_-(S_2) = \chi_-(S_1\oplus S_2)$ and the triangle inequality for the Thurston norm.
\end{proof}

Gabai in \cite{Ga1} introduced the notion of a sutured manifold and a decomposition surface.

\begin{definition}

		A \emph{sutured manifold} $(M,\gamma)$ is a compact oriented
3-manifold $M$ with boundary together with a set $\gamma \subset
\partial M$ of pairwise disjoint annuli $A(\gamma)$ and tori
$T(\gamma).$ Furthermore, the interior of each component of
$A(\gamma)$ contains a \emph{suture}, i.e., a homologically
nontrivial oriented simple closed curve. We denote the union of the
sutures by $s(\gamma)$ or $s(M).$ We say that an annulus component of $A(\gamma)$ is a \emph{sutured annulus} and a torus component of $T(\gamma)$ is a \emph{sutured torus}.

Finally every component of $R(\gamma)=\partial M \backslash
int(\gamma)$ is oriented. Define $R_+(\gamma)$ (or
$R_-(\gamma)$) to be those components of $\partial M \setminus
int(\gamma)$ whose normal vectors point out of (into) $M$.
The orientation on $R(\gamma)$ must be coherent with respect to
$s(\gamma),$ i.e., if $\delta$ is a component of $\partial
R(\gamma)$ and is given the boundary orientation, then $\delta$ must
represent the same homology class in $H_1(\gamma)$ as some suture.

\end{definition}

\begin{definition} \label{def:decompositionsurface}
	Let $(M,\gamma)$ be a sutured manifold, and $S$ a properly embedded surface in $M$ such that for every component $\lambda$ of $S\cap \gamma$ one of (1)-(3) holds:
	\begin{enumerate}
		\item $\lambda$ is a properly embedded nonseparating arc in an annular component of $\gamma$,
		\item $\lambda$ is a simple closed curve in an annular component $A$ of $\gamma$ in the same homology class as $A \cap s(\gamma)$,
		\item $\lambda$ is a homotopically nontrivial curve in a toral component $T$ of $\gamma$, and if $\delta$ is another component of $T\cap S$, then $\lambda$ and $\delta$ represent the same homology class in $H_1(T)$.
	\end{enumerate}
	Furthermore, no component of $\partial S$ bounds a disc in $R(\gamma)$ and no component of $S$ is a disc with $\partial D \subset R_+\cup R_-$.
	Then we call $S$ a \emph{decomposition surface} for $(M,\gamma)$. The orientation of $S$ induces orientations on $S_\pm \subset \partial \eta(S)$. We can construct a suture structure on $M' = \overline{M \backslash \eta(S)}$:
\begin{eqnarray*}
	\gamma' &=& (\gamma \cap M') \cup \overline{\eta(S'_+\cap R_-)} \cup \overline{\eta(S'_-\cap R_+)}\\
	R'_+ &=& ((R_+\cap M')\cup S'_+) \backslash \textnormal{int} \gamma'\\
	R'_+ &=& ((R_-\cap M')\cup S'_-) \backslash \textnormal{int} \gamma'
\end{eqnarray*} 
Here $S'_+$ (resp. $S_-'$) is the union of the components of $\partial\eta(S)$ whose normal vector points out of (resp. into) $M'$.

We say that $(M,R_+,R_-,\gamma) \stackrel{S} {\leadsto} (M',R'_+,R'_-,\gamma')$ is a \textit{sutured manifold decomposition}, and $(M,\gamma) \stackrel{S} {\leadsto} (M',\gamma')$ for short.

\end{definition}

We define a taut sutured manifold closely related to taut surfaces.

\begin{definition} \label{def:tautmanifold}
	A sutured manifold $(M,\gamma)$ is \textit{taut} if $M$ is irreducible and $R_+$ and $R_-$ are both taut surfaces in $M$. 
	\end{definition}

\begin{lemma}[{\cite[Lemma 0.4]{Ga2}}]
	Suppose $(M, \gamma)$ is a sutured manifold and $(M,\gamma) \stackrel{S} {\leadsto} (M',\gamma')$ is a sutured manifold decomposition. If $(M',\gamma')$ is taut, then so is $(M,\gamma)$.
\end{lemma}

\begin{definition}
	Let $(M,\gamma)$ be a sutured manifold. A \emph{product annulus} is an annulus $A$ in $M$ which does not cobound a solid cylinder in $M$ and such that one boundary component of $A$ lies on $R_-$ and the other one lies on $R_+$. A \emph{product disk} is a disk $D$ in $M$ such that $|D \cap s(\gamma)|= 2$.
\end{definition}

\begin{definition}
	Let $S \subset (M, \gamma)$ be a decomposition surface such that each component $J$ of $S$ is a product disk or a product annulus. Then we call $S$ a \emph{product decomposition surface}.
\end{definition}

Product disks and product annuli have a nice property of preserving the tautness.

\begin{lemma}[{\cite[Lemma 3.12]{Ga1}}]
	Let $(M, \gamma) \stackrel{J}{\leadsto}(M', \gamma')$ be a decomposition such that either $J$ is a product disc or a product annulus. Then $(M,\gamma)$ is taut if and only if $(M', \gamma')$ is taut.
\end{lemma}

Gabai \cite[Definition 4.11]{Ga1} associates to each connected sutured manifold $(M,\gamma)$ a complexity $c(M,\gamma)$ which is given by the 4-tuple $(c_1,c_2,c_3,c_4)$($c_i$ may be denoted $c_i(M,\gamma)$) where $c_i$ is a nonnegative integer, $c_2$ is a 6-tuple of nonnegative integers with the dictionary ordering, and $c _3$ and $c_4$ are finite, possibly empty, sets of positive integers. The set of complexities is given by the dictionary ordering.

\begin{definition} \label{reduced}
	Let $(M,\gamma)$ be a taut sutured manifold, and let $S$ be a maximal set of pairwise disjoint non-parallel product annuli and product disks in $M$.	
	Let $(\overline M,\overline \gamma)$ be the sutured manifold obtained from $(M,\gamma)$ by decomposing along $S$ and throwing away product sutured manifold components which we call \emph{windows}. Then $(\overline M,\overline \gamma)$ is called the \textit{guts} of $(M,\gamma)$ (which Gabai call the reduced sutured manifold of $(M,\gamma)$). The manifold $(\overline M,\overline \gamma)$ is well-defined up to isotopy by JSJ decomposition (\cite{jaco1979seifert},\cite[Chapter III]{johannson2006homotopy}). 
	
	Define the \textit{reduced complexity} $\overline C(M,\gamma)$ of the taut manifold $(M,\gamma)$ by $\overline C(M,\gamma) = c(\overline M,\overline \gamma)$.
\end{definition}

In this paper, we use reduced complexity rather than complexity of a sutured manifold and we briefly call reduced complexity as complexity.

Friedl and Kitayama \cite{FK} cited the following theorem in \cite[Section 4]{Ga1} which says the complexity of a taut sutured manifold will decrease or remain the same after decomposition.

\begin{theorem}\label{thm:complexityfunction}
Let $(M,\gamma)$ be a connected sutured manifold and let $(M,\gamma)\overset{S}{\rightsquigarrow} (M',\gamma')$ be a   sutured manifold decomposition along
 a    connected decomposition surface $S$.
Suppose that $(M,\gamma)$ and $(M',\gamma')$ are taut.  Let $(M'_0,\gamma'_0)$ be a component
of $(M',\gamma')$. Then
\[ \overline C(M'_0,\gamma'_0)\leq \overline C(M,\gamma).\]
Furthermore, if $S$ is not boundary parallel, e.g. if $[S]$ is non-trivial in $H_2(M,\partial M;\Z)$, then
\[ \overline C(M'_0,\gamma'_0)<\overline C(M,\gamma).\]
\end{theorem}

Among norm-minimizing and taut surfaces, we are more interested in classes with nice properties called properly norm-minimizing and properly taut. This is equivalent to 	(2) and (3) in Definition \ref{def:decompositionsurface} for decomposition surfaces.

\begin{definition} (Cf. \cite{miller2019effect}) 
	 Let $M$ be a compact 3-manifold whose boundary is a collection of tori. Let $S$ be an oriented surface in $M$. We say that $S$ is \emph{properly norm-minimizing} if $S$ is norm-minimizing and on each boundary component $P$ of $M$, all components of $\partial S\cap P$ have the same orientation (parallel on $P $). 
\end{definition}

\begin{definition}
Let $(M,R_+,R_-,\gamma)$ be a sutured manifold. We say that $(S,\partial S) \subset (M,\gamma)$ is \emph{properly taut} if $S$ is taut and on each component $P$ of $\gamma$, all components of $\partial S\cap P$ have the same orientation, and moreover, the orientation is the orientation of the core of $P$ if $P$ is an annulus.
\end{definition}

\begin{remark}
	When $M$ is a manifold with toral boundary, we think of $\gamma$ as $\partial M$ and $R_\pm$ is empty.
\end{remark}

\section{Guts Components as Sutured Manifolds}\label{sec:guts:guts}

In the following, we denote $M$ as a compact, orientable, connected, irreducible 3-manifold with toral boundary. Furthermore, we require $x_M$ to be non-degenerate. Consider $(M,\partial M)$ as a sutured manifold with $\gamma = \partial M$. 

\begin{definition}

	Suppose $z$ is a primitive element in $H_2(M,\partial M;\Z)$. Take a maximal collection of disjoint, non-parallel, 
	properly norm-minimizing decomposition surfaces $\Sigma_1,\ldots, \Sigma_k$ in $M$ such that 
$$
[\Sigma_i] = z \textnormal{ in }H_2(M,\partial M;\Z),  \ \forall i = 1,\ldots,k .
$$
We call the union of $\Sigma_1,\ldots, \Sigma_k$ a \emph{facet surface} for $z$ and denote $\Sigma_1\cup\ldots \cup\Sigma_k$ as $F(z)$. 

\end{definition}

\begin{definition}\label{def:guts}

Let $S$ be a properly norm-minimizing surface in $M$ and consider $M\spl S$ as a sutured manifold. We define the \emph{guts} of $S$ as the guts of $M\spl S$ in Definition \ref{reduced}, denoted by $\Gamma(M,S)$.

\end{definition}

\begin{definition}

	\emph{Equivalence of guts.} Two guts $G_1$ and $G_2$ are equivalent if there exists an isotopy of $M$ which restricts to a homeomorphism $G_1 \to G_2$ as sutured manifolds.

\end{definition}

We are going to show some properties of guts.

\begin{definition}
	A properly embedded surface $S$ in a sutured manifold $ (M,\gamma)$ is called a horizontal surface if
i) $\partial S \subset \gamma$ and $\partial S$ is isotopic to $\partial R_+(\gamma)$,
ii) $S$ is taut,
iii) $[S] = [R_+(\gamma)]$ in $H_2(M, \gamma)$.
\end{definition}

\begin{definition}[{\cite[Definition 9.3]{juhasz2008floer}}]
	A sutured manifold is called \emph{reduced} if it doesn't contain any essential product annulus or product disk, and \emph{horizontally prime} if it doesn't contain any horizontal surface which is not parallel to boundary.
	
\end{definition}

\begin{lemma} \label{lem:reducedhorizontallyprime}
For a facet surface $F$, each component of $\Gamma(M,F)$ is reduced and horizontally prime.

\end{lemma}

\begin{proof}
First, the reason that each component $(G,\gamma)$ of $\Gamma(M,F)$ is reduced is because we decompose along a maximal collection of product decomposition surfaces. Next, suppose that there exists a taut surface $\Sigma$ in $G$ which is not parallel to the boundary of $G$ but homologous to $R_+$ and $R_-$. Take a union $F'$ of some components of $F$ such that $F'$ contains the $R_+$ part of $(G,\gamma)$ and represents a multiple of $[F]$ in $H_2(M,\partial M)$. We cut $R_+$ out of $F'$ and glue with $\Sigma$ to have a new surface $F''$. Then $F''$ is homologous to a multiple of $F$, disjoint with $F$ and is non-parallel to $F$ which violates the maximality of $F$. 

\end{proof}

\begin{lemma}\label{lem:atmostonegut}  Let $F$ be a facet surface in $M$. Then every connected component of $M \spl F$ contains at most one guts component.
\end{lemma}

\begin{proof}
Suppose a component of $M \spl F$ contains at least 2 guts components. Denote the component as $N$ and consider it as a sutured manifold $(N, R_+(N),R_-(N), \gamma(N))$. Let $(G,R_+(G),R_-(G), \gamma(G))$ be a guts component in $(N, R_+(N),R_-(N), \gamma(N))$. Then $R_\pm(G)$ is a subset of $R_\pm(N)$. 

Take a union $F'$ of some components of $F$ such that $F'$ contains the $R_+(N)$ part of $N$ and represents a multiple of $[F]$ in $H_2(M,\partial M)$. We cut $R_+(G)$ out of $F'$ and glue with $R_-(G)$ to have a new surface $F''$. Then $F''$ is homologous to a multiple of $F$, disjoint with $F$ and is non-parallel to $F$ which violates the maximality of $F$. See Figure \ref{fig:2-guts}.
 
 \begin{figure}[hbt]
	
  \includegraphics[width=0.7\columnwidth]{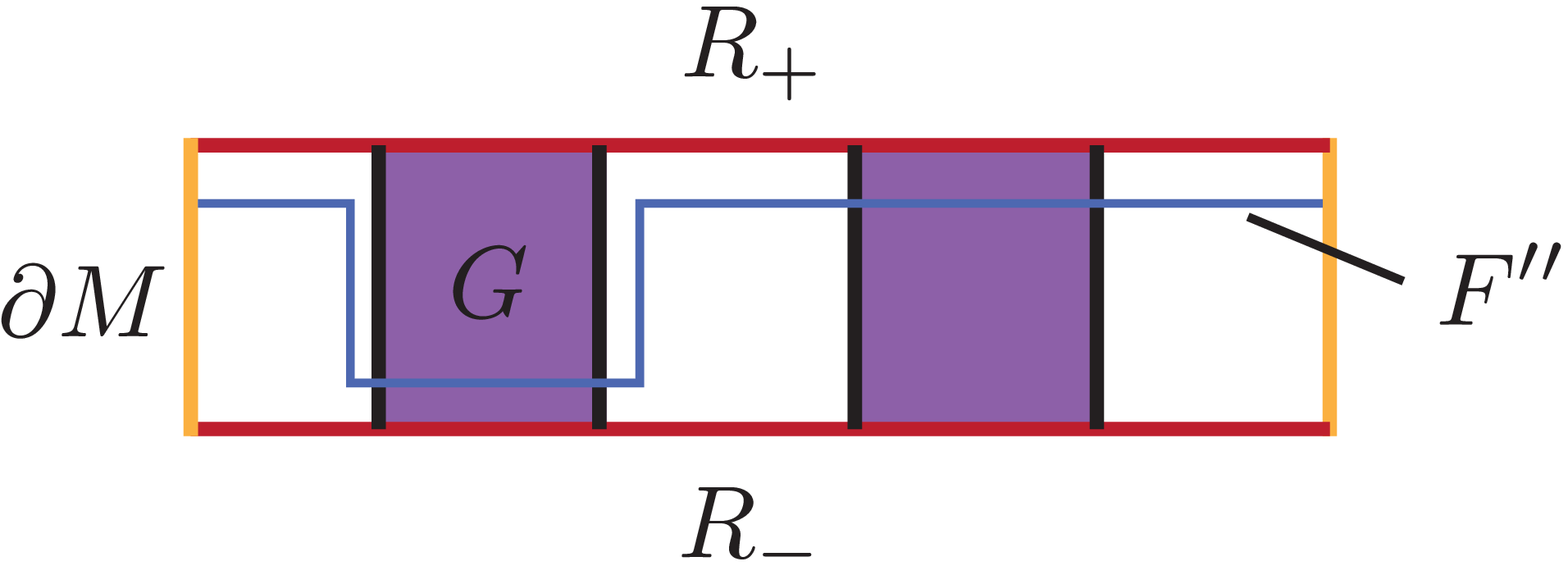}
    \caption{} \label{fig:2-guts}
\end{figure}
 
\end{proof}

For the guts associated to facet surfaces representing an element in $H_2(M,\partial M)$, we have the following theorem.

\begin{namedthm*}{Theorem \ref{thm:gutsforhomology}}
	Let $M$ be an irreducible, orientable $3$-manifold with boundary a disjoint union of tori $\sqcup_{i=1}^n P_i$ and non-degenerate Thurston norm. The guts $\Gamma(M,F)$ for facet surfaces representing a primitive element $z$ in $H_2(M,\partial M;\Z)$ do not depend on the selection of facet surfaces and product decomposition surfaces up to the equivalence of guts.
\end{namedthm*}

Before proving the theorem, we need some prerequisites.

\begin{definition} \label{def:non-product}
Let $(M,R_+,R_-)$ be a sutured manifold. Let $A$ be an essential annulus in $M$ with boundary on $R_+\cup R_-$. We say $A$ is a \emph{non-product annulus} if it is neither a product annulus nor boundary-parallel to $R_+\cup R_-$.
\end{definition}
\begin{remark} \label{rmk:non-product}
	Let $S$ be a facet surface and $A$ be a non-product annulus in $M\spl S$. Let $N$ be the component of $M\spl S$ which contains $A$. We can do a surgery along $A$ for $S$ to have $S'$ as in Figure \ref{fig:non-product}.
	\begin{figure}[hbt]
	\begin{subfigure}[b]{.45\textwidth}
	\begin{center}
  \includegraphics[width=\columnwidth]{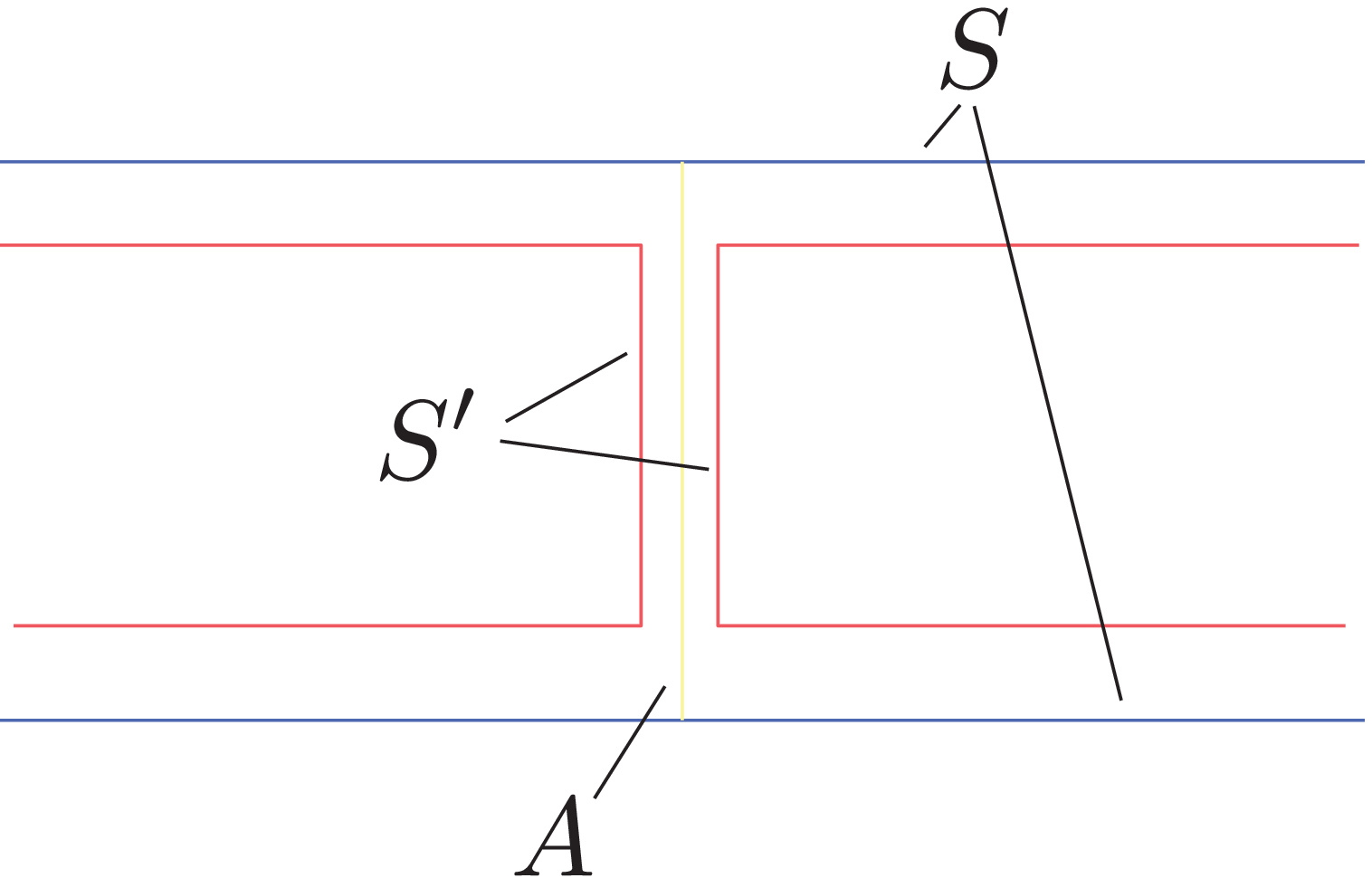}
  \caption{A schematic graph of a non-product annulus.}\label{fig:non-product}
\end{center}
  \end{subfigure}%
  \hfill
  \begin{subfigure}[b]{.45\textwidth}
  	\begin{center}
  \includegraphics[width=\columnwidth]{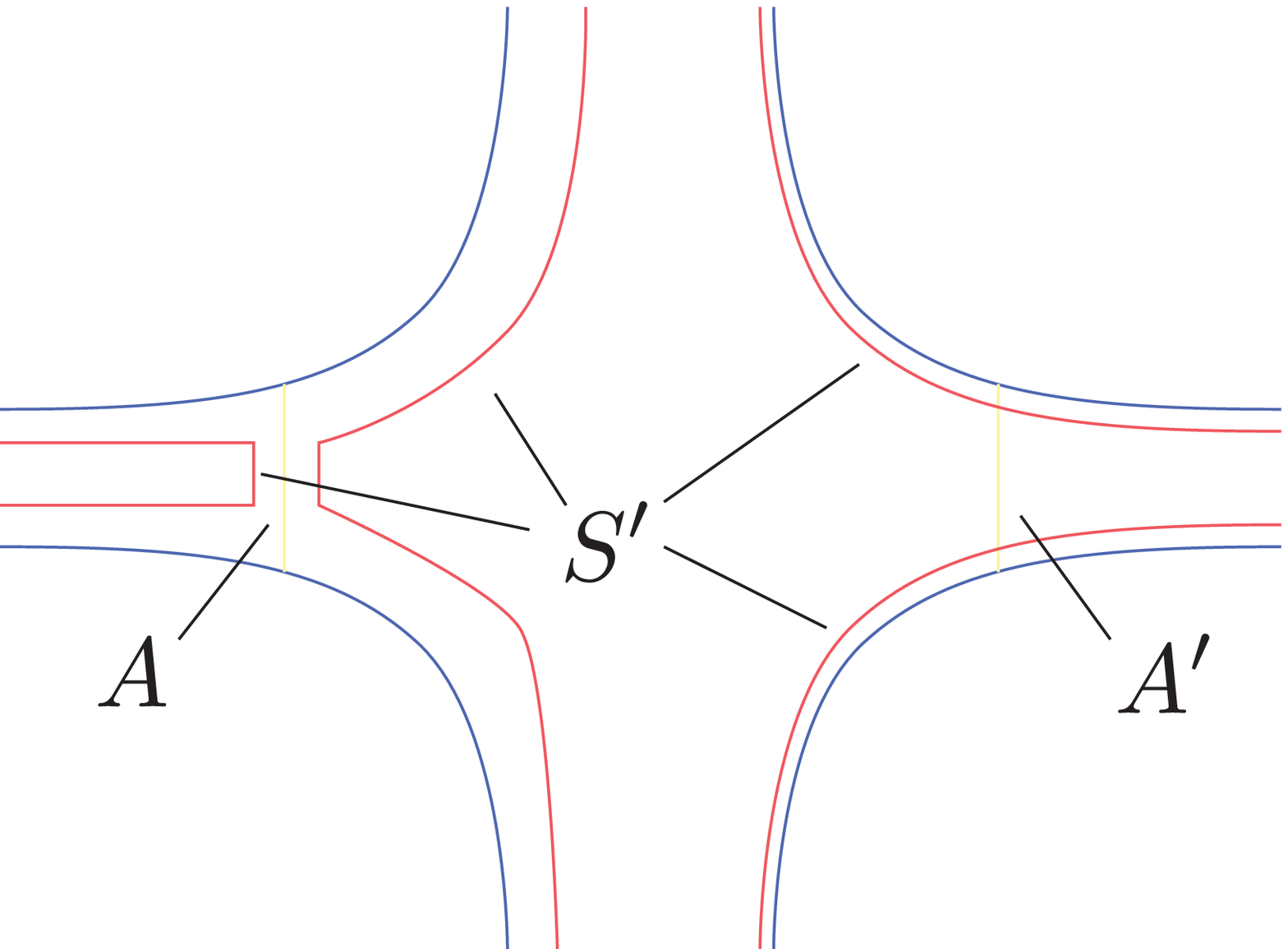}
  \caption{A schematic graph of 2 non-parallel non-product annuli}\label{fig:2-non-product}
  \end{center}
  \end{subfigure}%
  \caption{}
\end{figure}

	 There are two possibilities:
	\begin{enumerate}
		\item $S'$ has no null-homologous component. In this case, the guts of $N$ is a solid torus with 4 longitudinal sutures which we call a \emph{4-ST} in this paper,
		\item $S'$ has null-homologous components. In this case, $A$ is null-homologous rel boundary in $N$. Furthermore, the guts of $N$ is a sutured manifold with toral boundary, only two sutures on a boundary component and the other boundary components as sutured tori.
	\end{enumerate}
\end{remark}

\begin{lemma} \label{lem:2-non-product} Let $S$ be a facet surface. Then non-product annuli in a component $N$ of $M\spl S$ are parallel to each other up to orientation.
\end{lemma}

\begin{proof}
	Suppose there are two non-product annuli $A$ and $A'$ which are non-parallel. See Figure \ref{fig:2-non-product}. We do a surgery along $A$ for $S$ and drop the (possibly) null-homologous components to have $S'$. Do the operation again along $A'$ for $S'$ to have $S''$. Then by the previous remark, between $S$ and $S'$ there is a guts component and so is between $S'$ and $S''$. This violates Lemma \ref{lem:atmostonegut}. 
\end{proof}

\begin{definition} \label{def:belly} Let $S$ and $T$ be two facet surface representing a primitive class $z$ in $H_2(M,\partial M)$. Furthermore, $S$ and $T$ are in minimal position. We call a subsurface $Q \subset S\cup T$ from decomposing $ S\cup T$ along $S\cap T$ a \emph{patch}. For a boundary component $P$ of a component $J$ of $M\spl (S\cup T)$, if the co-orientation of $P \cap S$ is the same as the co-orientation of $P \cap T$, we call $P$ a \emph{bad boundary} and $J$ a \emph{bad region}. See Figure \ref{fig:badsolid} as an example.

\begin{figure}[hbt]
\begin{subfigure}[t]{.45\textwidth}
	\begin{center}
  \includegraphics[width=\columnwidth]{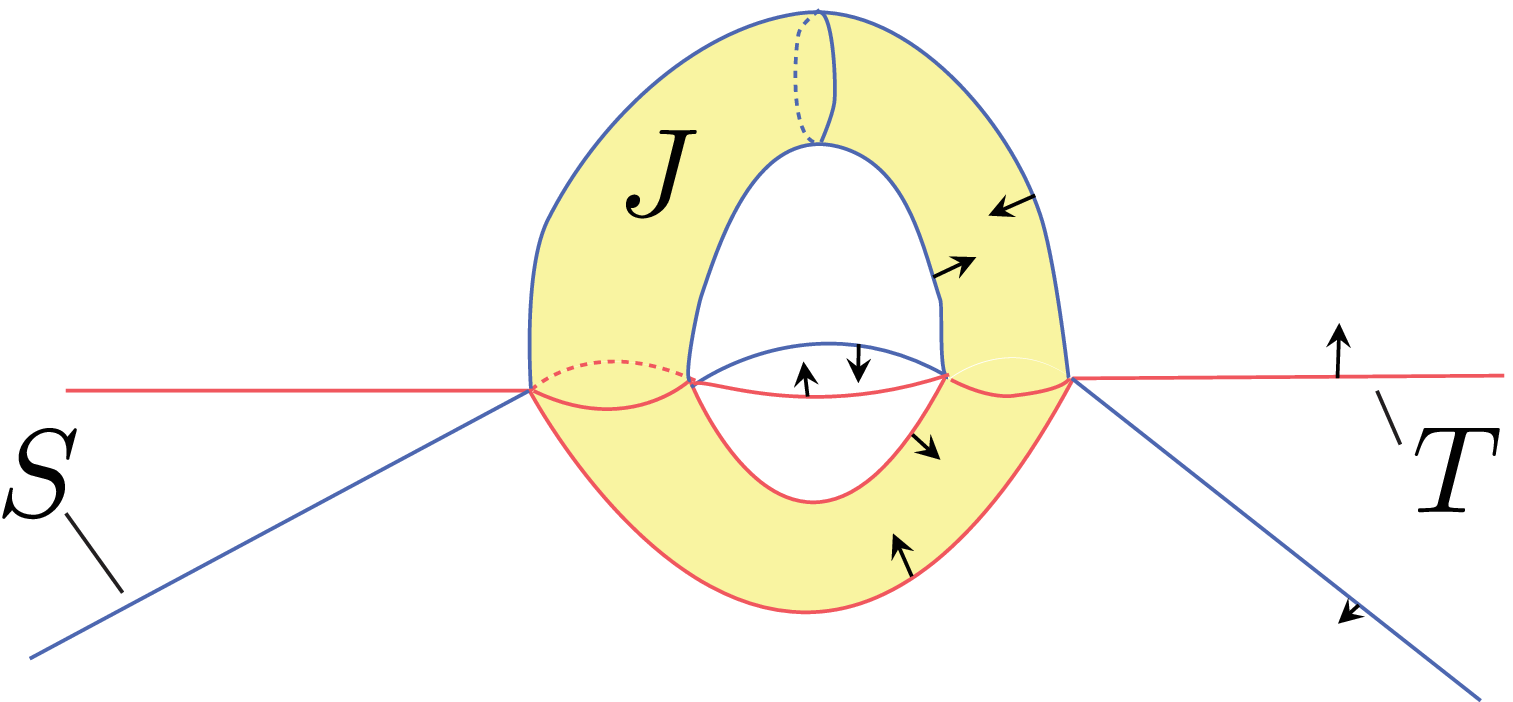}
  \caption{An example of a bad region and a bad boundary.}\label{fig:badsolid}
\end{center}
  \end{subfigure}%
  \hfill
  \begin{subfigure}[t]{.45\textwidth}
  	\begin{center}
  \includegraphics[width=\columnwidth]{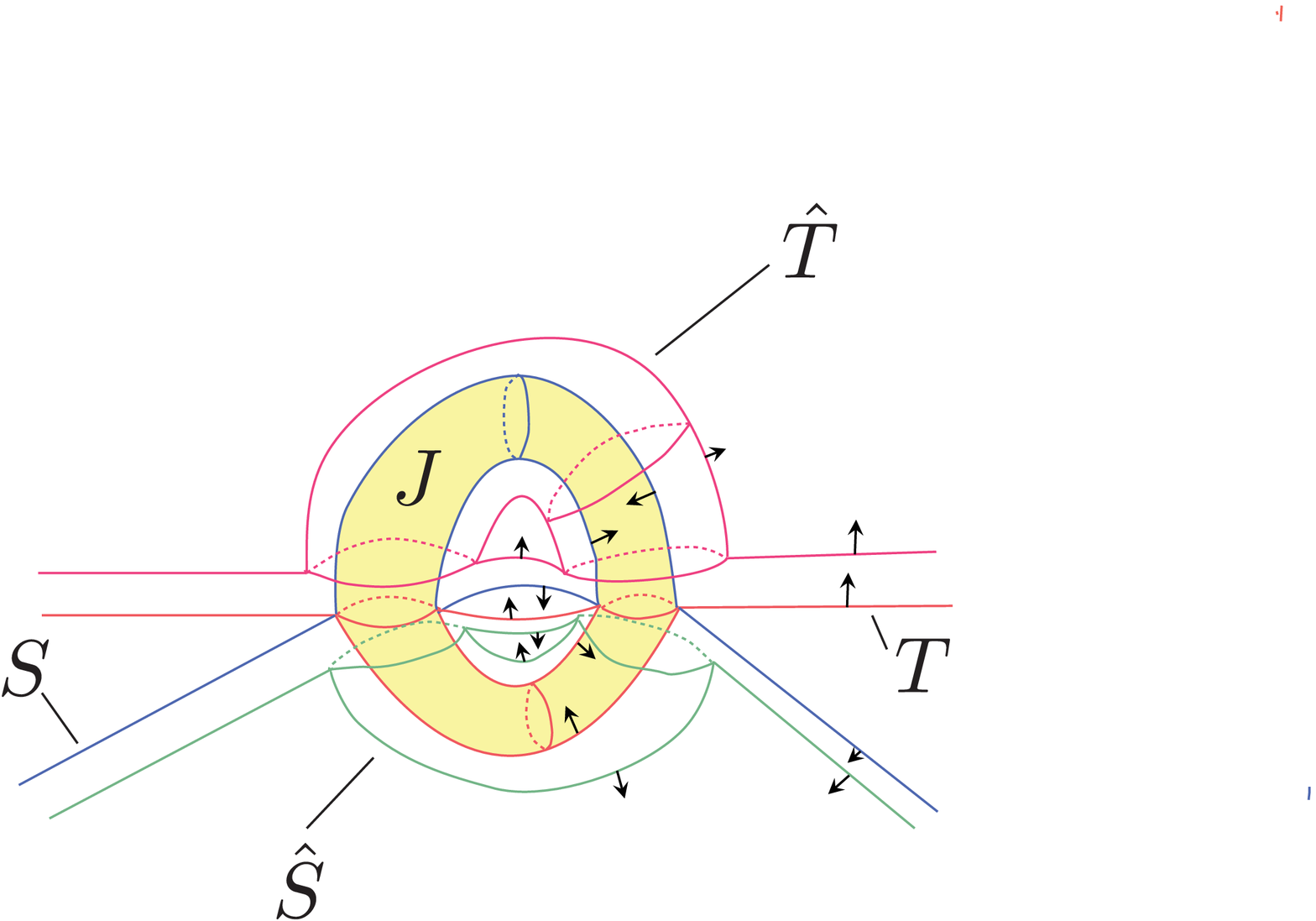}
  \caption{}\label{fig:belly}
  \end{center}
  \end{subfigure}%
  \caption{}
\end{figure}

\end{definition}

\begin{lemma} \label{lem:bad-boundary}
	Under the same hypotheses as Definition \ref{def:belly}, for a bad boundary $P$, $P$ is a torus and it must satisfy one of the following conditions:
	\begin{enumerate}
		\item it consists of only two patches, or
		\item it consists of 4 patches and the bad region $J$ whose boundary contains $P$ is a 4-ST.
	\end{enumerate}
\end{lemma}

\begin{proof}
	By the definition of the double curve sum, $P$ is a subset of $S\oplus T$. From Lemma \ref{lem:dcs}, $S \oplus T$ minimizes the Thurston norm. Since $\partial J$ is null-homologous, $P$ is a torus. 
	
	Suppose $P$ contains 4 or more patches. We denote the patches on $P$ as \[ U_1,V_1,U_2,V_2,\ldots, U_k,V_k\] in order where $U_i$ is on $P \cap S$ and $V_i$ is on $P \cap T$. Then each $V_i$ is a non-product annulus in $M\spl S$ and by Lemma \ref{lem:2-non-product}, $V_1,V_2,\ldots, V_k$ are parallel to each other. In the neighborhood $I \times A$ (where $A$ is an annulus) of $V_1,V_2,\ldots, V_k$, each $V_i$ is homeomorphic to a copy of $A$. Since $J$ is connected, it can only contain two $V_i$ and hence $J$ is a solid torus with 4 longitudinal patches on $\partial J$.	
	
\end{proof}

We actually need a slight generalization of the previous lemma. 

\begin{lemma}\label{lem:4-ST-2}
	Under the same hypotheses as Definition \ref{def:belly}, let $N$ be a component of $M\spl S$ and $P_1,\ldots,P_k$ are the bad boundaries with respect to $S$ and $T$ in $N$. Then there are $k$ regions of $N\spl T$ such that the guts of each is a 4-ST which can be isotoped to a neighborhood of $P_i$ and furthermore, its suture is isotoped to a patch of $P_i$.
\end{lemma}

\begin{proof}
We already prove the case when $k = 1$. So we assume $k \ge 2$.

	By Lemma \ref{lem:2-non-product}, we know that the components of $P_i \cap T$ are parallel to each other up to orientation. Furthermore, since there is at most one bad boundary which consists of only two patches, we can assume $P_1,\ldots,P_k$ are already in order as in each case of Figure \ref{fig:bad-boundary-k-all}. Furthermore $P_1 \cap T, P_2 \cap T,\ldots,P_k \cap T$ are contained in a solid torus whose boundary is the union of 4 annuli $U_1,V_1,U_2,V_2$ where $U_i$ is a subset of $S$, $V_1$ is a patch of $P_1 \cap T$ and $V_2$ is a patch of $P_k\cap T$.	 Then for each $i$, between $P_i$ and $P_{i+1}$, there exists a 4-ST. 	
	
	\begin{figure}[hbt]
	\begin{subfigure}[t]{.5\textwidth}
	\begin{center}
  \includegraphics[width=2.5 in]{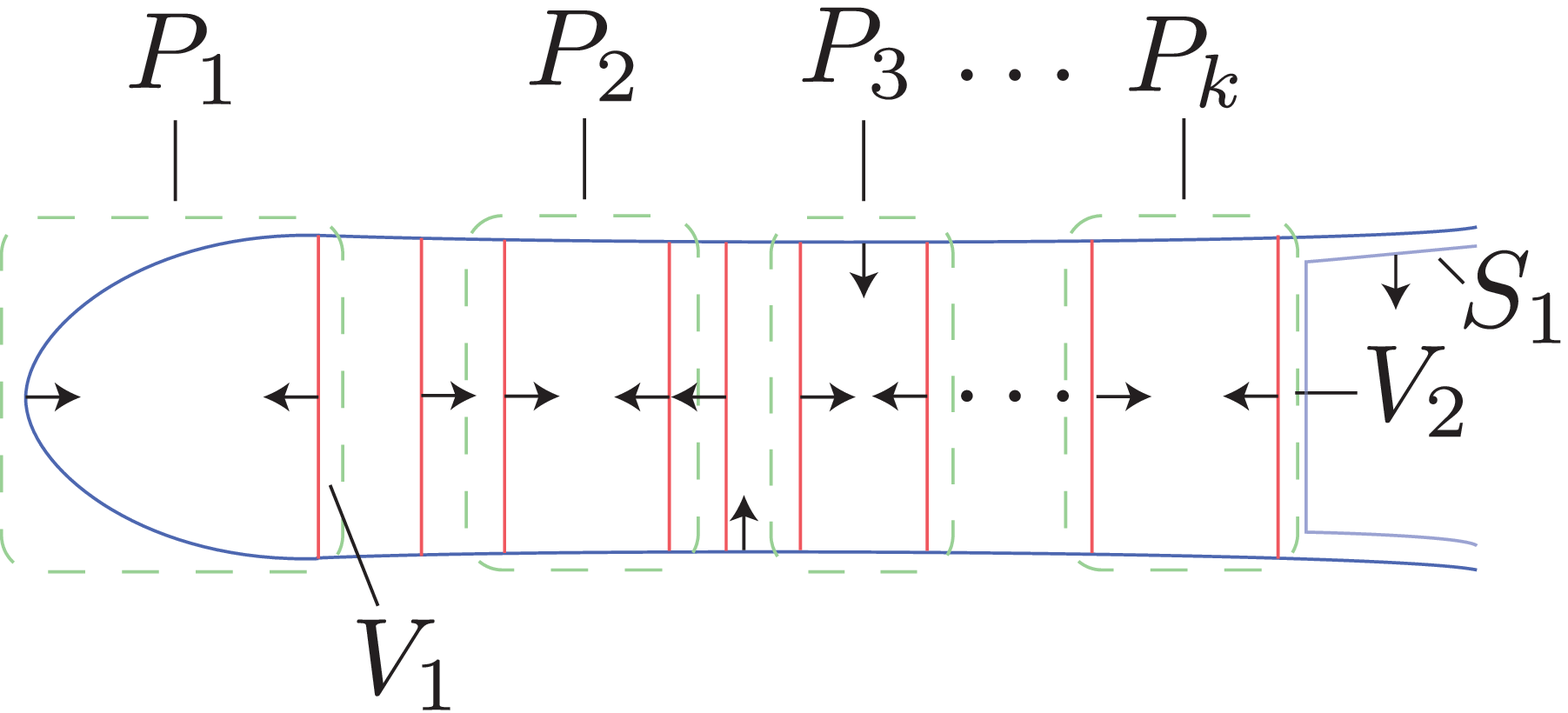}
  \caption{First case.}\label{fig:bad-boundary-2-k}
\end{center}
  \end{subfigure}%
  \begin{subfigure}[t]{.5\textwidth}
	\begin{center}
  \includegraphics[width=2.5 in]{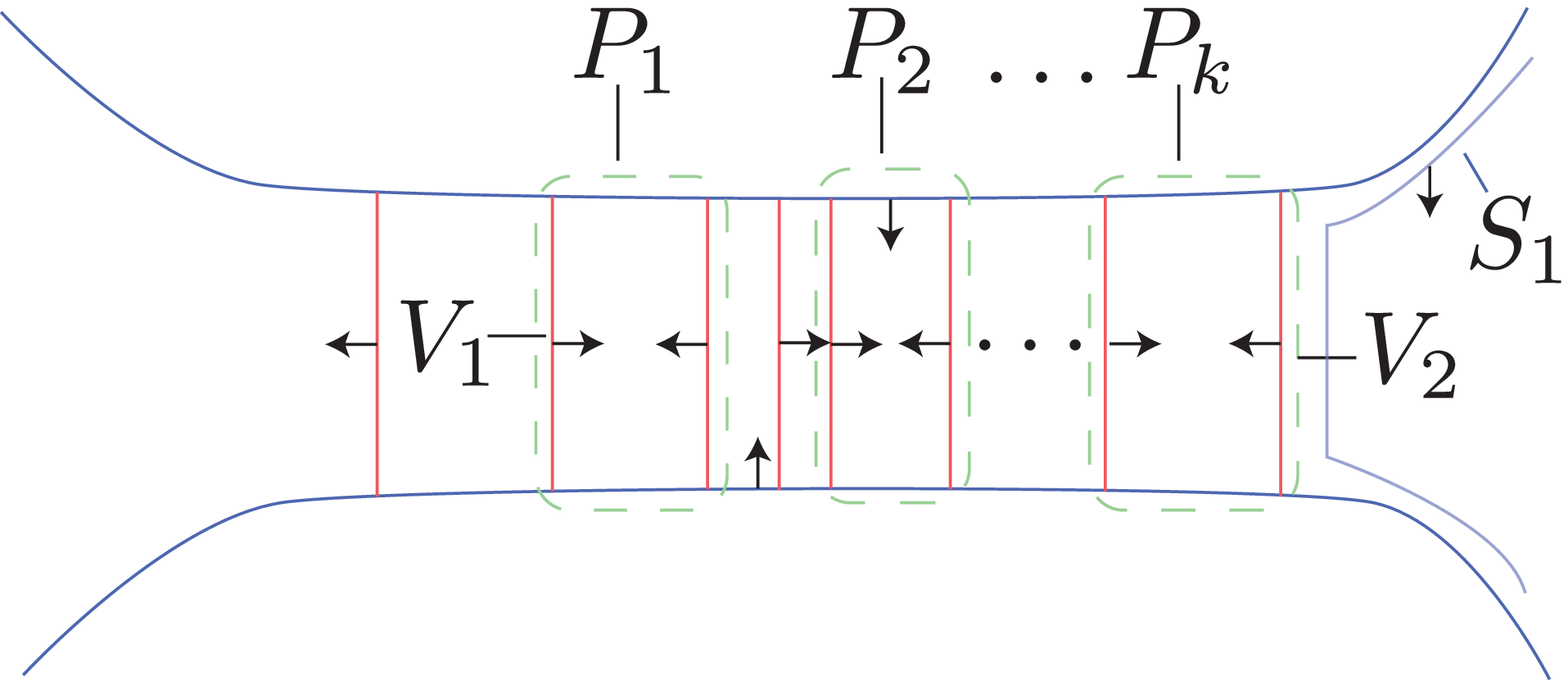}
  \caption{Second case.}\label{fig:bad-boundary-k}
  \end{center}
  \end{subfigure}%

  \caption{$S$ is in blue and the non-product annuli components of $T\cap N$ are in red.}\label{fig:bad-boundary-k-all}
  \end{figure}

	Let $S_0$ be the union of components of $S$ which intersect $V_1 \cup V_2$. We remove $U_1$ and $U_2$ from $S_0$ and attach $V_1$ and $V_2$ to have a homologous surface $S'_0$. Let $S'_i$ denote the component of $S'_0$ which contains $V_i$. Since $S_0$ is homologically nontrivial, at least one of $S'_1$ and $S'_2$ is also homologically nontrivial. Without loss of generality, we assume $S'_2$ is homologically nontrivial. Then by the maximality of $S$, $S'_2$ is isotopic to a component $S_1$ of $S$. Furthermore, $P_k$ can not consist of only two patches because $S'_2$ is not a torus. Thus by Lemma \ref{lem:bad-boundary}, $P_k$ consists of 4 patches and it is the boundary of a 4-ST. Then the region $J$ bounded by $V_2$ and $S_1$ has a 4-ST guts. Even if there are components of $T\cap N$ inside $J$, they do not intersect the guts. Hence we can find a region of $N\spl T$ inside $J$ such that its guts is a 4-ST. 
	
	Since we can assign the 4-ST between $P_i$ and $P_{i+1}$ to $P_i$, we find $k$ distinct regions that satisfy the condition.

	\end{proof}

\begin{definition} Under the same hypotheses as Definition \ref{def:belly}, let $\Gamma$ be the collection of the guts components of $M \spl (S\cup T)$. By Lemma \ref{lem:4-ST-2}, for each bad region, there is a 4-ST guts component of $\Gamma$ in the same component of $M\spl S$. We remove the bad regions and their corresponding 4-ST guts components from $\Gamma$. Then we call the remaining guts components the \emph{belly} of $S$ and $T$ and denote it as $\B(S,T)$. 
\end{definition}

\begin{lemma}
	Let $S$ and $T$ be two facet surface representing a primitive class $z$ in $H_2(M,\partial M)$. Furthermore, $S$ and $T$ are in minimal position.  Then $\B(S,T)$ is equivalent to $\B(T,S)$ as sutured manifolds up to isotopy.
\end{lemma}

\begin{proof}
The only difference between $\B(S,T)$ and $\B(T,S)$ is that we remove different equivalent 4-ST's related to the bad regions. That does not affect the equivalence of $\B(S,T)$ and $\B(T,S)$.
\end{proof}

\begin{remark}  \label{rmk:belly-def}
In the definition of the belly $\B(S,T)$, we remove the bad regions and their corresponding 4-ST guts components from $\Gamma$. This is equivalent to say that if we have two parallel non-product annuli with co-orientation pointing each other, we can remove one of them. Hence after this operation, we do not have two parallel non-product annuli with co-orientation pointing each other. We further remove duplicate parallel non-product with same orientation since they do not affect the guts of the resulting manifold from decomposing $M \spl S$ along the remaining patches of $T \spl S$. Hence for each group of parallel non-product annuli (without orientation) of $T \spl S$, there remain at most two non-product annuli. If there remain two non-product annuli, their co-orientations point outward and they bound a 4-ST. If we remove any one of them, the guts does not change. 

Therefore, we can replace each group of parallel non-product annuli (without orientation) of $T \spl S$ by only one non-product annulus to have $T'$ such that the resulting sutured manifold from decomposing $M \spl S$ along $T'$ has no bad boundary and its guts is equivalent to $\B(S,T)$. From now on, we use this definition for $\B(S,T)$. 
\end{remark}

\begin{lemma} \label{lem:belly}
	Let $S$ and $T$ be two facet surfaces representing a primitive class $z$ in $H_2(M,\partial M)$. Furthermore, $S$ and $T$ are in minimal position. Then the guts with respect to $S$ is equivalent to the belly $\B(S,T)$.
\end{lemma}

\begin{proof}
By Lemma \ref{lem:dcs}, $S\oplus T$ minimizes the Thurston norm among surfaces representing $[S]+[T]$ (though it might contains null-homologous tori).

Let $T_0$ be the union of patches of $T\spl S$ with negative Euler characteristic and $N$ be $M \spl S$. We want to show that after the sutured decomposition along $T_0$, $N \spl T_0$ is taut.

Since $S\oplus T$ minimizes the Thurston norm, $N \spl T_0$ is irreducible. We suppose a component $N_0$ of $N \spl T_0$ is not taut. Without loss of generality, we assume $R_+(N_0)$ is not taut. Then there is a surface $R'$ that is homologous to $R_+(N_0)$ in $H_2(N_0, \partial R_+(N_0))$ and has smaller Thurston norm. We can do a cut-and-paste surgery on $S\oplus T$ that removes $R_+(N_0)$ and pastes back $R'$. The resulting surface is homologous to $S\oplus T$ but with smaller Thurston norm, which is a contradiction.



By \cite[Lemma 3.2]{Sc2}, there exists a positive integer $k$ such that we can isotope $kS \oplus T$ away from $S$. Furthermore, as described in the proof of the lemma, the isotopy only moves the part of $kS \oplus T$ which is inside a tubular neighborhood $\eta(S)$ of $S$ and fix the rest. Note that $T_0$ is $T \spl S $. Then during the isotopy, $T_0$ is always inside $N$. 

We drop the null-homologous components of $kS \oplus T$ and denote the rest as $T'$. Because $T$ is homologous to $S$, $T'$ is homologous to a multiple of $S$. Since $S$ is a facet surface, we know that each component of $T'$ is boundary parallel to $N$. Here, we say a surface $U$ is boundary parallel if there is an isotopy of $U$ onto the boundary of $N$ where we allow moving the boundary of $U$. Since each component of $T_0$ has negative Euler characteristic, $T'$ still contains $T_0$ and hence $T_0$ is boundary parallel in $N$ too. 

Let $N_i$ be a component of $N$ and $T_{0,i} \stackrel{\Delta}{=} T_0 \cap N_i$. We denote the isotopy of $T_{0,i}$ onto the boundary of $N_i$ as $f_t: T_{0,i} \to N_i, t \in [0,1]$ where $f_0 = \mathds{1}$ and $T'_{0,i} \stackrel{\Delta}{=} f_1(T_{0,i}) \subset \partial N_i$. Because the isotopy of $T_{0,i}$ comes from the isotopy of $T'$, the images $f_t(T_{0,i})$ for different $t$ do not intersect each other outside of $\partial N_i$. Hence the isotopy restricted on the unfixed components of $\partial T_{0,i}$ induces a collection $A_i$ of embedded annuli and $\partial A_i \subset \partial T_{0,i} \cup \partial T'_{0,i}$. See Figure \ref{fig:isotopy-all} for a schemetic point of view. We consider two cases of $A_i$:

\begin{figure}[hbt]
	\begin{subfigure}[t]{.5\textwidth}
	\begin{center}
  \includegraphics[width=2.5 in]{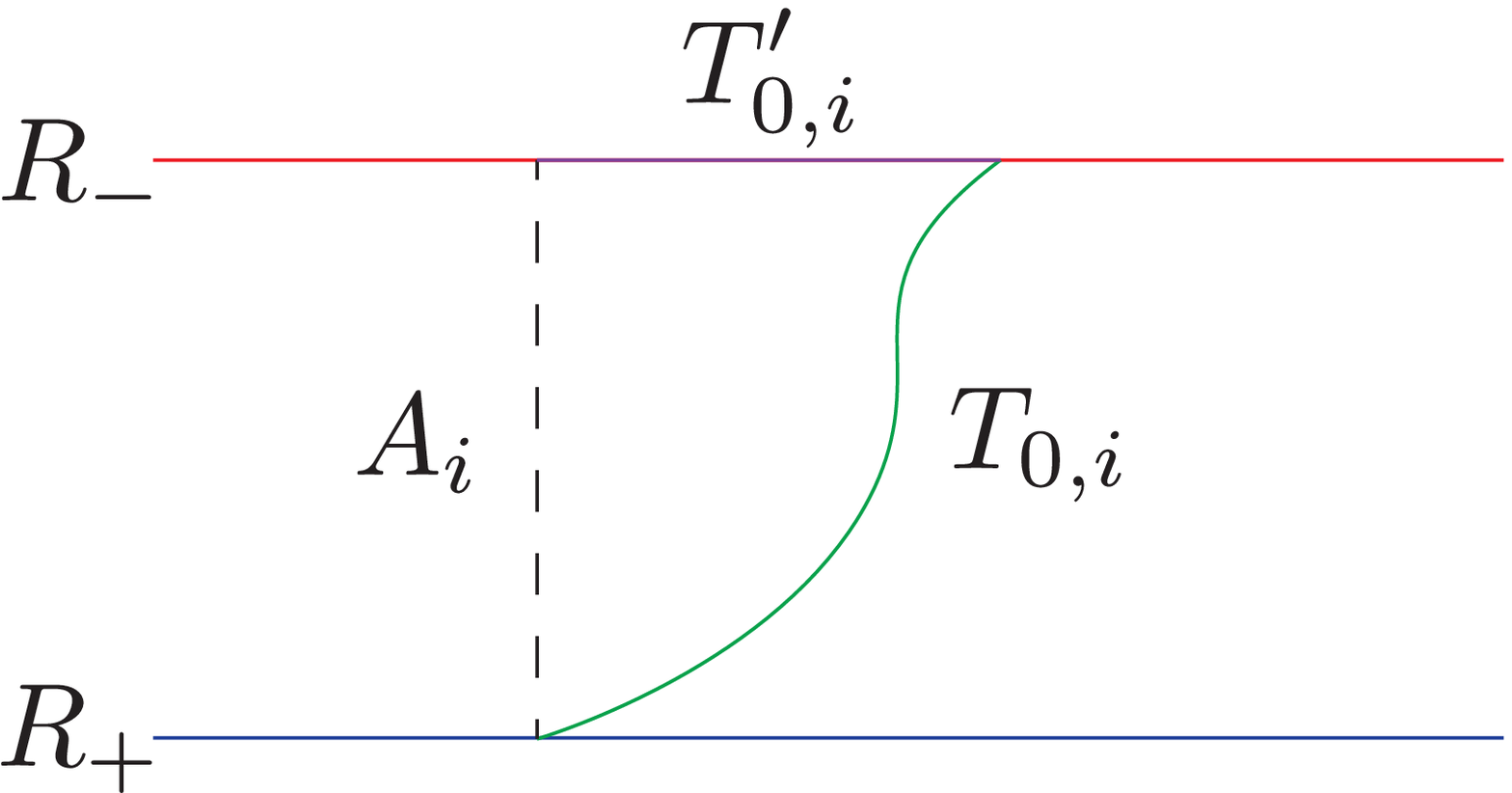}
  \caption{First case.}\label{fig:isotopy}
\end{center}
  \end{subfigure}%
  \begin{subfigure}[t]{.5\textwidth}
	\begin{center}
  \includegraphics[width=2.5 in]{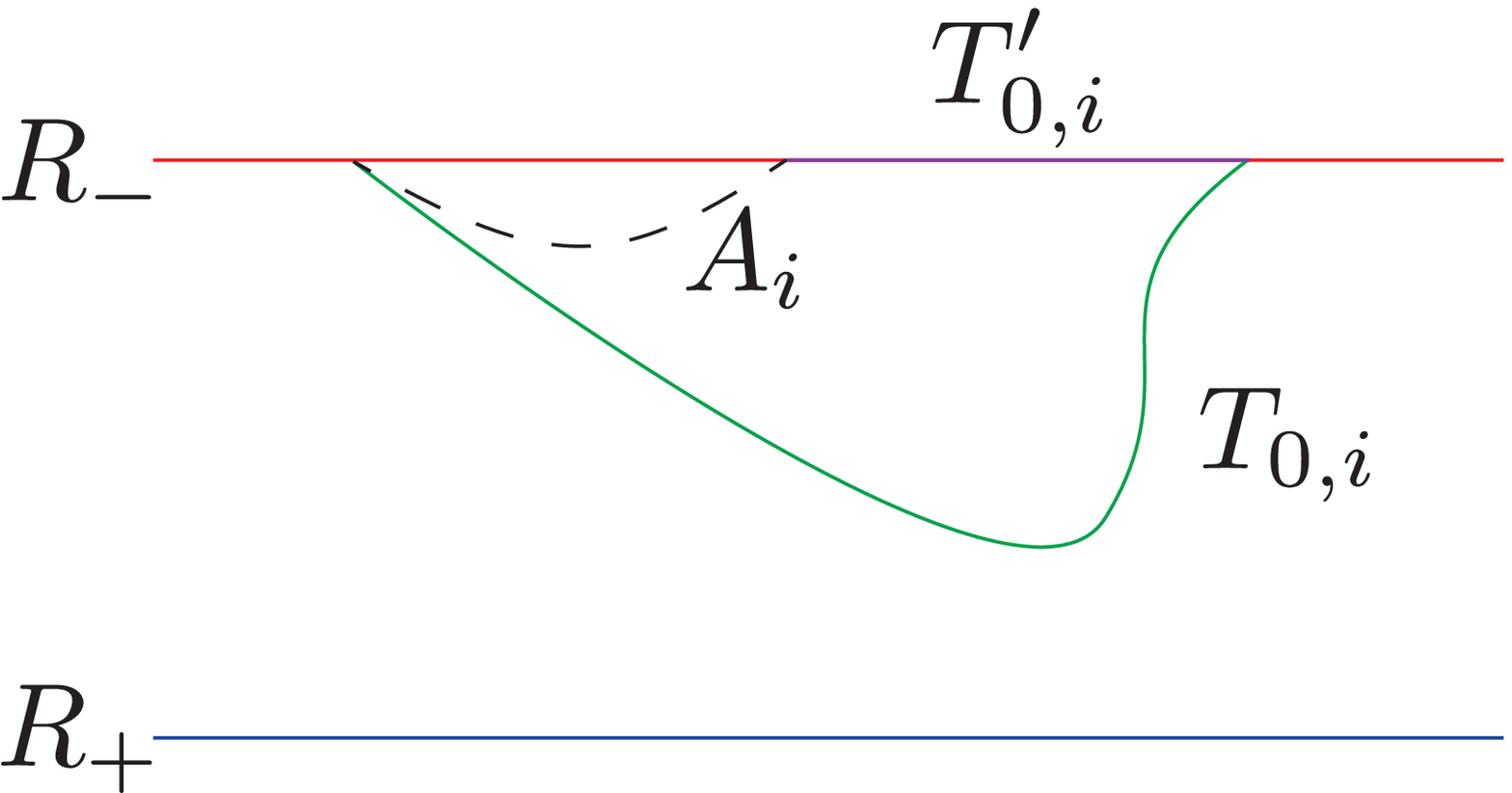}
  \caption{Second case.}\label{fig:isotopy-2}
  \end{center}
  \end{subfigure}%

  \caption{A shematic view of the isotopy of $T_{0,i}$. $R_-$ and $R_+$ of $N_i$ are in red and blue, respectively. $T_{0,i}$ is in green and $T'_{0,i}$ is in purple.} \label{fig:isotopy-all}
  \end{figure}

\begin{enumerate}[topsep=0.5ex]
	\item $A_i$ only consists of product annuli or boundary parallel annuli. See Figure \ref{fig:isotopy}. For each component $T_{0,i,j}$ of $T_{0,i}$, we let $A_{i,j}$ be the union of components of $A_i$ induced from $\partial T_{0,i,j}$ and $N'_{i,j}$ be the resulting sutured manifold from decomposing $N_i$ along the essential components of $A_{i,j}$. In order to show that decomposing $N_i$ along $T_{0,i,j}$ does not change the guts of $N_i$, we only need to show that the guts of decomposing along $T_{0,i,j} \cup A_{i,j}$ is equivalent to the guts of decomposing along $A_{i,j}$, which is equivalent to the guts of $N_i$. We think of the decomposition along $T_{0,i,j} \cup A_{i,j}$ as first decomposing $N_i$ along $A_{i,j}$ and then decomposing $N'_{i,j}$ along $T_{0,i,j} \spl A_{i,j}$. Since $T_{0,i,j} \spl A_{i,j}$ is boundary parallel in $N'_{i,j}$ with boundary on the sutured annuli of $N'_{i,j}$, we know that decomposing $N'_{i,j}$ along $T_{0,i,j} \spl A_{i,j}$ only creates product sutured manifold components. Thus we prove that decomposing $N_i$ along $T_{0,i}$ does not change the guts of $N_i$. 
	\item $A_i$ contains an essential non-product annulus $A_{0,i}$. See Figure \ref{fig:isotopy-2}. Then, as in Remark \ref{rmk:non-product}, we know that the guts $G_i$ of $N_i$ is a 4-ST or a sutured manifold with toral boundary and 2 sutures on a boundary component. Let $C_{0,i}$ be the boundary component of $A_{0,i}$ in $T_{0,i}$. Because each component of $T_{0,i}$ either contains $C_{0,i}$ or is disjoint with it, the algebraic intersection number between any component of $T_{0,i}$ and $C_{0,i}$ is 0. Since each suture of $G_i$ is isotopic to $C_{0,i}$, the algebraic intersection number between any component of $T_{0,i}$ and any suture of $G_i$ is 0 too. Noting that the boundary of $G_i$ is a torus, $T_{0,i} \cap \partial G_i$ are isotopic to the sutures of $G_i$. Hence we can isotope $T_{0,i}$ so that $T_{0,i}$ only intersects $\partial G_i$ on $\gamma(G_i)$. Now we analyze the two possiblities of $G_i$ separately.
	\begin{enumerate}
	\item If $G_i$ is a 4-ST, we know that $T_{0,i} \cap G_i$ consists of annuli parallel to $\partial G_i$ and we can isotope $T_{0,i} \cap G_i$ into the tubular neighborhood of $R_+(G_i)$ and $R_-(G_i)$ in $N_i$. 
	\item On the other hand, suppose $G_i$ is a sutured manifold with toral boundary, only two sutures on a boundary component and the other boundary components as sutured tori. Since $T_{0,i}$ is null-homologous in $H_2(N_i,\partial N_i)$, $T_{0,i} \cap G_i$ is also null-homologous in $H_2(G_i,\partial G_i)$, i.e., $T_{0,i} \cap G_i$ is homologous to $R_+(G_i)$ and $R_-(G_i)$ rel $\gamma(G_i)$. Note that, by Lemma \ref{lem:reducedhorizontallyprime}, $G_i$ is horizontally prime. Hence $T_{0,i} \cap G_i$ is parallel to $R_+(G_i)$ or $R_-(G_i)$ and thus can be isotoped into the tubular neighborhood of $R_+(G_i)$ and $R_-(G_i)$ in $N_i$. 
	\end{enumerate}
	
	 \begin{figure}[hbt]
	
  \includegraphics[width=0.7\columnwidth]{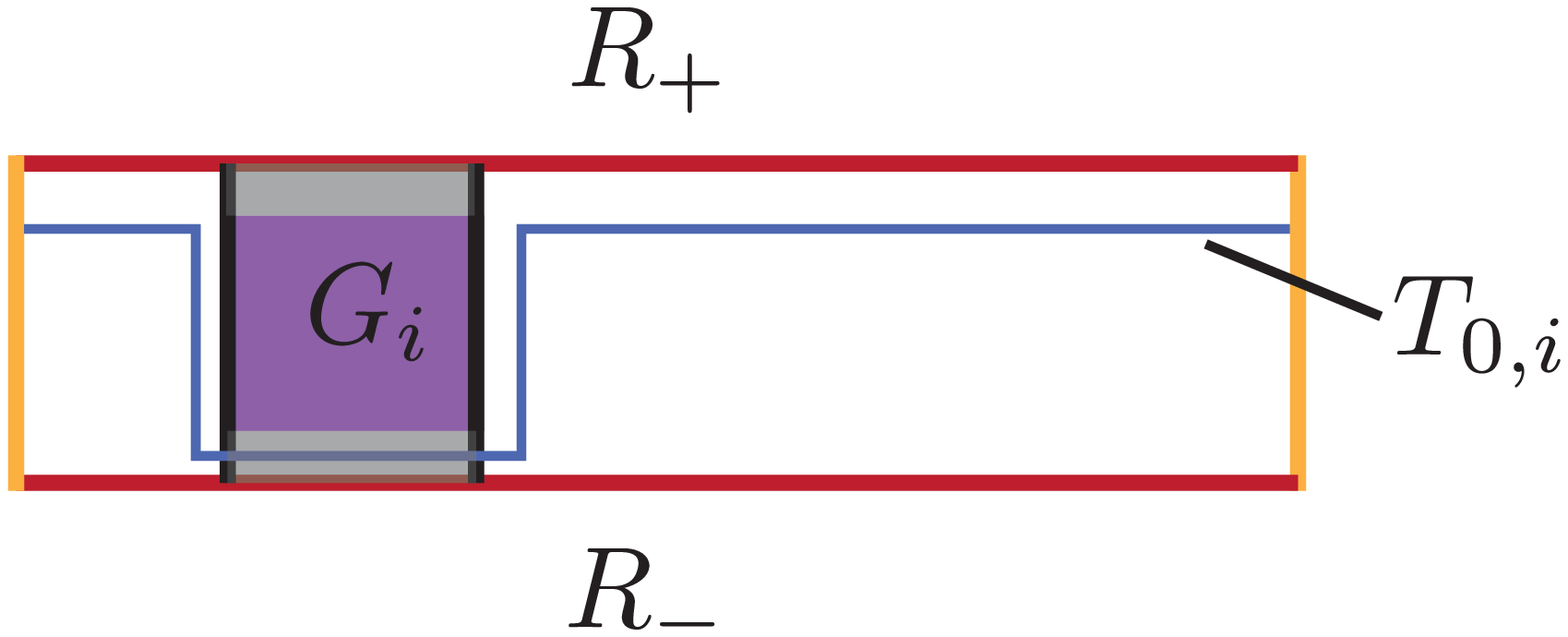}
    \caption{$T_{0,i} \cap G_i$ is isotoped into the tubular neighborhood of $R_+(G_i)$ and $R_-(G_i)$ in $N_i$, which is in gray. } \label{fig:isotopy-gut}
\end{figure}

Therefore in each of these two cases, we can isotope $T_{0,i} \cap G_i$ into the tubular neighborhood of $R_+(G_i)$ and $R_-(G_i)$ in $N_i$, as in Figure \ref{fig:isotopy-gut}. Since decomposing $N_i$ along $T_{0,i}$ obtains the guts from decomposing $G_i$ along $T_{0,i} \cap G_i$, it only creates product sutured manifolds and does not change the guts of $N_i$.

\end{enumerate}

In conclusion, the guts of $N$ is equivalent to the guts of $N\spl T_0$.

Next, we consider annuli patches of $T\spl S$. If an annulus patch is a product annulus or has a boundary component on $\partial M$, decomposing along it does not affect the guts. So we focus on non-product annuli. 

We replace each group of parallel non-product annuli of $T \cap (M \spl S)$ by only one non-product annuli with a specific orientation as in Remark \ref{rmk:belly-def}. Let $A_1,\ldots,A_k$ be the non-parallel non-product annuli. $A_i$ does not form a bad boundary with $S$. 

Let $N_i$ be the component of $M\spl S$ containing $A_i$. Without loss of generality, we assume the boundary components of $A_i$ are on $R_-(N_i)$. 

We consider separately for the two possible cases.

First, we suppose that the resulting surface from decomposing $R_-(N_i)$ along $\partial A_i$ contains an annulus $A'$. Then the union of $A'$ and $A_i$ is a torus. Since $A_i$ does not form a bad boundary with $S$, $A_i$ is homologous to $A'$ with orientation. Then we can replace $A'$ with $A_i$ from $R_-(N_i)$ to have a norm-minimizing surface $R'$ homologous to $R_-(N_i)$. By the maximality of $S$, $R'$ is parallel to a subsurface of $R_-(N_i) \cup R_+(N_i)$. Since the guts of the resulting manifold from decomposing $N_i$ along $A_i$ is equivalent to the guts of the resulting manifold from decomposing $N_i$ along $R'$, decomposing along $A_i$ does not change the guts of $N_i$.

Second, we suppose each component of the resulting surface from decomposing $R_-(N_i)$ along $\partial A_i$ has negative Euler characteristic. We do a surgery for $R_-(N_i)$ along $A_i$ to have a norm-minimizing surface $R'$ homologous to $R_-(N_i)$. Similarly, $R'$ is parallel to $R_-(N_i) \cup R_+(N_i)$. Furthermore, the components of $R'$ that contains a copy of $A_i$ are parallel to $R_+(N_i)$. Therefore the guts of $N_i$ contains a 4-ST and because of Lemma \ref{lem:atmostonegut}, the guts of $N_i$ is actually a 4-ST. Furthermore, after decomposing $N_i$ along $A_i$, the guts still remains as a 4-ST. See Figure \ref{fig:NT} for reference.

	\begin{figure}[hbt]
	\begin{center}
		\includegraphics[width=3 in]{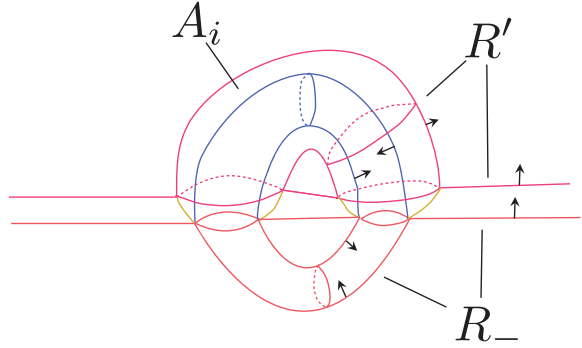}
  \caption{Illustration of surfaces inside $N_i$. $R_-$ is in red, $A_i$ is in blue and $R'$ is in pink.}\label{fig:NT}
	\end{center}
	  
  \end{figure}

In summary, decomposing $M\spl S$ along $A_1 \cup \ldots \cup A_k$ also does not change the guts. Since we already consider all possibilities, decomposing $M\spl S$ along each patch of $T\spl S$ does not change the guts of $M\spl S$. The theorem is proved.

\end{proof}

\begin{proof}[Proof of Theorem~\ref{thm:gutsforhomology}]
	Let $S, T$ be 2 facet surfaces representing $z$ and isotope $S,T$ to be in minimal position. By Lemma \ref{lem:belly}, $\Gamma(M,S)$ is equivalent to $\B(S,T)$ and $\Gamma(M,T)$ is equivalent to $\B(T,S)$. Hence by Lemma \ref{lem:belly}, $\Gamma(M,S)$ is equivalent to $\Gamma(M,T)$.
\end{proof}

\begin{definition} \label{def:gutsforhomology}
	For an element $z$ in $H_2(M,\partial M;\Z)$, we define $(\Gamma(z),\gamma(z))$ (or just $\Gamma(z)$) to be the guts with respect to facet surfaces representing $z$ in Theorem \ref{thm:gutsforhomology}.
	\end{definition}

\begin{definition}
	Given a taut sutured manifold $(M,R_+,R_-,\gamma)$, we do a sutured decomposition along a maximal collection of horizontal surfaces and then decompose along a maximal product decomposition surfaces. By throwing away the components which are product sutured manifolds, we call the remaining sutured manifolds \emph{horizontally prime guts} of $(M,\gamma)$.
\end{definition}

Using the same method in the proof of Theorem \ref{thm:gutsforhomology}, we are able to show 

\begin{corollary}\label{cor:sutured guts} 
Let $(M,\gamma)$ be a taut sutured manifold such that there is no non-separating annulus or torus rel $\gamma$. 
The horizontally prime guts of $(M,\gamma)$ is well defined up to equivalence.
\end{corollary}

\section{Guts and Thurston Cones}\label{sec:guts:cone}

In this section, we are going to prove Theorem \ref{thm:samegutsboundarycase}.

\begin{definition} \label{def:guts:restriction}
	Let $M$ be a $3$-manifold with boundary a disjoint union of tori $\sqcup_{i=1}^n P_i$ and $\alpha$ be an element in $H_2(M,\partial M;\R)$. Let $\partial: H_2(M,\partial M;\R) \to H_1(\partial M;\R)$ be the map from the long exact sequence of $(M,\partial M)$. We denote the restriction of $\alpha$ on a boundary component $P_i$ as the coordinate of $\partial \alpha$ in $H_1(P_i;\R)$. We say two non-zero elements $u_1, u_2 \in H_1(P_i;\R)$ are in opposite orientation if $u_1$ is a negative multiple of $u_2$.
\end{definition}

\begin{namedthm*}{Theorem \ref{thm:samegutsboundarycase}}
Let $M$ be an irreducible, orientable $3$-manifold with boundary a disjoint union of tori $\sqcup_{i=1}^n P_i$ and non-degenerate Thurston norm. Let $y,z$ be two elements in an open face of the Thurston sphere. If there is an open segment $(v,w)$ containing $y,z$ in the open face such that the restrictions of $v$ and $w$ on each boundary component are not in opposite orientations, the guts $\Gamma(y)$ is equivalent to $\Gamma(z)$.
\end{namedthm*}

Before proving the theorem, we need a slight modification of \cite[Proposition 2.8]{miller2019effect} and a lemma.

\begin{proposition} \label{prop:maggie}

	Let $\alpha_1$ and $\alpha_2$ be in the same closed Thurston face of a non-degenerate Thurston norm $x_M$, where $M$ is a $3$-manifold with boundary a disjoint union of tori $\sqcup_{i=1}^n P_i$. If the restriction of $\alpha_1$ and $\alpha_2$ on each boundary component $P_j$ are not in opposite orientation, then there exist properly norm-minimizing surfaces $S_1$ and $S_2$ with $[S_i]=\alpha_i$ so that for any positive integers $a$ and $b$, $aS_1 \oplus bS_2$ is properly norm-minimizing.

\end{proposition}
\begin{proof} (Cf. the proof of \cite[Proposition 2.8]{miller2019effect})
	Since $x_M$ is nondegenerate, $\chi_-(R)=-\chi(R)$ for any norm-minimizing surface $R$ in $M$.

Let $S_1$ and $S_2$ be properly norm-minimizing surfaces with $[S_i]=\alpha_i$. Isotope the $S_i$ near their boundaries so that $\partial S_1$ and $\partial S_2$ intersect minimally. Because the restriction of $\alpha_1$ on each boundary component $P_i$ is not a negative multiple of the restriction of $\alpha_2$, we know that for each $j$, every component of $\partial (aS_1 \oplus bS_2) \cap P_j$ is nontrivial and have the same orientation which means it represents the same homology class. We have $x_M(a[S_1]+b[S_2])=ax_M([S_1])+bx_M([S_2])=a\chi_-(S_1)+b\chi_-(S_2)=-a\chi(S_1)-b\chi(S_2)=-\chi(aS_1\oplus bS_2)$. Then we are done if $aS_1\oplus bS_2$ has no disk, $2$-sphere, torus or annulus components.

Suppose $S_1\setminus S_2$ includes a component $C$ which does not meet $\partial S_1$ and with $\chi(C)\ge 0$. Then surger $S_2$ along $C$ and throw away null-homologous components to obtain $S'_2$ (i.e. $S'_2:=[S_2\setminus((\partial C)\times I)]\cup(C\times S^0)$). The surface $S'_2$ is homologous to $S_2$. Since $\chi(C)\ge 0$, $S'_2$ is norm-minimizing. Set $S_2:=S'_2$ and repeat until $S_1\setminus S_2$ includes no such component $C$. Now any closed components of $aS_1 \oplus bS_2$ must include a region homeomorphic to some component of $S_1\setminus S_2$, which must have negative Euler characteristic. Therefore, $aS_1 \oplus bS_2$ has no closed sphere or torus components.

Since every component of $\partial (aS_1 \oplus bS_2) \cap P_j$ is nontrivial, $aS_1\oplus bS_2$ has no disk components. Suppose $aS_1\oplus bS_2$ has an annulus component $A$, by the non-degeneracy of $x_M$, $A$ is homologically trivial in $H_2(M,\partial M)$ which means $\partial A$ is on a boundary component $P_j$ and is homologically trivial in $H_1(P_j)$ which violates the fact that the two components of $\partial A$ have the same orientation.

\end{proof}

We state a simple lemma without proof.
\begin{lemma} \label{lem:productsuturedmanifold}
	Let $(N,R_+,R_-,\gamma)$ be a product sutured manifold with $\chi(N) < 0$. If $S$ is a properly embedded surface without annulus components such that $S$ is homologous to $R_+$ in $H_2(N,\gamma)$ and $\chi_-(S) = \chi_-(R+)$, then $S$ is parallel to $R_+$.
\end{lemma}

\begin{proof}[Proof of Theorem \ref{thm:samegutsboundarycase}]
Let $y$ and $z$ be in the same open Thurston cone and $(v,w)$ be an open segment containing the closed segment $[y,z]$ in the same cone such that the restrictions of $v$ and $w$ on each boundary component are not in opposite orientations. By Proposition \ref{prop:maggie}, we can take properly norm-minimizing representatives $S_v$ and $S_w$ for $v$ and $w$ respectively such that $S_v\backslash S_w$ and $S_w\backslash S_v$ has no component of nonnegative Euler characteristic. Moreover, for any positive integers $a$ and $b$, $aS_v \oplus bS_w$ is properly norm-minimizing. 

Similar to \cite[Proposition 8]{oertel1986homology}, we can create a properly taut oriented branched surface $ B$ carrying $S_v$ and $S_w$ by changing each curve of $S_v \cap S_w$ to an ``annulus or disk of contact". $ B$ also carries $aS_v \oplus bS_w$ for any positive integers $a$ and $b$. Let $N_B$ be the fibered neighborhood of $B$ and we decompose $\partial N_B$ into two parts: horizontal boundary $\partial_h N_B$ which is transversal to the fibers and vertical boundary $\partial_v N_B$ which consists of some fibers. See Figure \ref{fig:branched-surface}.

	\begin{figure}[hbt]
	\begin{center}
		\includegraphics[width=3 in]{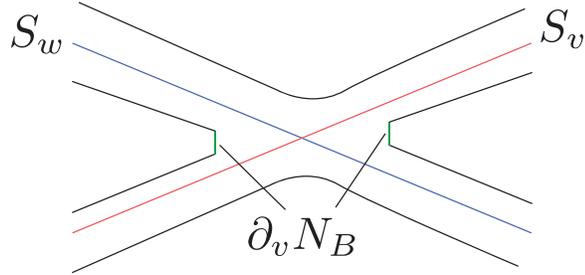}
  \caption{Schematic figure for the branched surface $B$.}\label{fig:branched-surface}
	\end{center}
	  
  \end{figure}
  
Let $S$ be the surface $aS_v \oplus bS_w$ for some $a$ and $b$. As in \cite[Section 6]{tollefson1996taut}, we isotope $2S$ so that it contains $int_{2S}(\partial_h N_B)$ where $int_{2S}(\partial_h N_B)$ is the interior of $\partial_h N_B$ in $2S$ and let $L_S$ denote the $I$-bundle obtained by splitting $N_B$ along $2S - int_{2S}(\partial_h N_B)$. Then the vertical boundary of $L_S$ is the union of $\partial M$ and $\partial_v N_B$.

 We want to show no component of $L_S$ is $K \times I$, where $K$ is a punctured disk in the interior of $2S$ and there is a boundary curve $\alpha$ of $K$ such that $\alpha \times \{0\}$, $\alpha \times \{1\}$ bound disks $D_0, D_1$ in $2S$ containing $K \times \{0\},K \times \{1\}$, respectively. Otherwise, suppose there is such a component $K\times I$. Since $K$ is in the interior of $2S$, $\partial D_0$ is in $\partial_v N_B$. Because $\partial_v N_B$ is isotopic to $(S_u \cap S_w) \times I$, we can isotope $D_0$ such that $\partial D_0$ is a component of $S_v \cap S_w$. Since $S_v$ is incompressible, $\partial D_0$ bounds a disk in $S_v$. This means there is a disk component in $S_v\backslash S_w$, which is a contradiction to the construction of $S_v$ in Proposition \ref{prop:maggie}.

We can also obtain $L_S$ as follows. First, we decompose $M$ along $2S$ and then further decompose along $\partial_v N_B$. The resulting manifold is the union of $L_S$ and $M\backslash N_B$. Let $A$ be an annulus component of $\partial_v N_B$. Then by the preceding paragraph $\partial A$ is nontrivial in $2S$ and hence $A$ is a product annulus. Therefore $\partial_v N_B$ is a union of some product annuli and product disks.

Let $F$ be a facet surface that represents $av+bw$ and contains $S$. Isotope $2F$ such that it intersects minimally the product annuli and product disks that separate $L_S$ and $M\backslash N_B$. For each component $J$ of $L_S$, we think of it as a product sutured manifold $(J,R_+,R_-,\gamma)$. We want to show that, $2F$ intersects each component $\gamma_i$ of $\gamma$ in the same orientation. Since $2F$ is properly norm-minimizing, $2F \cap J$ is a taut surface homologous to a multiple of $[R_+]$. By Lemma \ref{lem:productsuturedmanifold}, $2F \cap J$ is parallel to $R_+$ and hence $2F$ intersects $\gamma$ in the same orientation.

We want to show that $2F \cap (M \backslash N_B)$ is a maximal collection of horizontal surfaces (with multiplicities) for $M \backslash N_B$. Since $2F$ intersects each of $\partial_v N_B$ and $\partial M$ in the same direction, $2F \cap (M \backslash N_B)$ is a union of horizontal surfaces for $M \backslash N_B$. If there is a horizontal surface $F_0$ in a component $(J,R_+,R_-,\gamma)$ of $M \backslash N_B$ which is disjoint with and not parallel to $2F \cap (M \backslash N_B)$, we can remove $R_+$ from $2F$ and attach $F_0$ to the rest to have a surface which is disjoint with and not parallel to $2F$. This violates the maximality of $F$.

 Hence the guts of $av+bw$ is equivalent to the horizontally prime guts of $M \backslash N_B$. By Corollary \ref{cor:sutured guts}, the guts of $av+bw$ remain the same for different choice of $a$ and $b$. Since we can choose $a$ and $b$ such that $av+bw$ is a positive multiple of $y$ and $z$ respectively, $\Gamma(y)$ is equivalent to $\Gamma(z)$.

\end{proof}

For the case when $M$ is closed, any open segment $(v,w)$ containing $y,z$ in the open face satisfies the condition that the restrictions of $v$ and $w$ on each boundary component are not in opposite orientations. Hence we have the following corollary.

\begin{corollary}
	When $M$ is closed, let $y,z$ be two elements in an open face of the Thurston sphere. Then the guts $\Gamma(y)$ are equivalent to $\Gamma(z)$.
\end{corollary}

\section{Restriction Map to the Guts in Homology}\label{sec:guts:map}

The following theorem is a generalization of what was used in Agol's paper about virtual fibering conjecture \cite[Lemma 4.1]{Ag1}.

\begin{namedthm*}{Theorem \ref{T2}}
		Let $z$ lie on a $k$-codimensional open Thurston cone $\Delta$. We denote the subspace spanned by $\Delta$ as $V$. Then the kernel of the restriction map $\varphi:H_2(M,\partial M;\Z) \to H_2(\Gamma(z),\partial \Gamma(z);\Z)$ is a subspace of $V$.
	
	If there is a neighborhood of $z$ in $\Delta$ which does not contain two elements whose restrictions on a boundary component are in opposite orientations, the kernel of $\varphi$ is exactly $V$ and the rank of the image of $\varphi$ is $k$, i.e. the rank of the image of 
	$$  H^1(M) \to H^1(\Gamma(z))$$
	is $k$.
\end{namedthm*}

Before proving the theorem, we need several useful lemmas:
\begin{lemma}\label{avoid}
Let $z$ be an element in $H_2(M,\partial M;\Z)$ and $\varphi$ be the restriction map $\varphi:H_2(M,\partial M;\Z) \to H_2(\Gamma(z),\partial \Gamma(z);\Z)$.
	The image of $u$ under the map $\varphi$ is 0 if and only if we can find a surface $\Sigma$ representing a multiple of $u$ in $H_2(M,\partial M;\Z)$ avoids the guts of $z$.
\end{lemma}

\begin{proof}
	First, we assume $\partial M = \emptyset$.
	Looking at the long exact sequence for $(M, M')$ where $M' \stackrel{\Delta}{=} M\backslash \interior{\Gamma(z)} $:
	$$\cdots \to H_2(M') \stackrel{i_*}{\to} H_2(M) {\to} H_2(M,M') \to  H_1(M') \to \cdots $$
	
	we have a short exact sequence
		$$H_2(M') \stackrel{i_*}{\to} H_2(M) \stackrel{\varphi}{\to} H_2(\Gamma(z),\partial \Gamma(z))$$
		
	because of the excision lemma $ H_2(M,M') \cong  H_2(\Gamma(z),\partial \Gamma(z))$.
	
	Hence $\varphi(u) = 0$ if and only if there exists a surface $\Sigma$ in $M'$ such that $i_*([\Sigma]) = u$ and hence $\varphi(i_*([\Sigma])) = 0$, which means $\Sigma$ avoids the guts of $z$.
	
	When $\partial M \ne \emptyset$, we have a long exact sequence 
	\[
	\cdots \to H^1(M,\Gamma(z)) {\to} H^1(M) \stackrel{\varphi}{\to} H^1(\Gamma(z)) \to \cdots
	\]
	and the excision lemma
	\[
	H^1(M,\Gamma(z)) \cong H^1(M\backslash \interior{\Gamma(z)},\partial \Gamma(z)\backslash \partial M ).
	\]
	
	Hence we have the following short exact sequence:
	\[
	H^1(M\backslash \interior{\Gamma(z)},\partial \Gamma(z)\backslash \partial M ) {\to} H^1(M) \stackrel{\varphi}{\to} H^1(\Gamma(z))	
	\]

	Noting that $\partial (M \backslash \interior{\Gamma(z)}) = (\partial \Gamma(z)\backslash \partial M) \cup (\partial M \backslash \partial \Gamma(z))$, by Poincare duality for manifolds with boundary, we have $H^1(M\backslash \interior{\Gamma(z)},\partial \Gamma(z)\backslash \partial M ) \cong H_2(M\backslash \interior{\Gamma(z)}, \partial M \backslash \partial \Gamma(z))$. Also by Poincare duality, we have $H^1(M) \cong H_2(M,\partial M)$ and $H^1(\Gamma(z)) \cong H_2(\Gamma(z), \partial \Gamma(z))$.

	 Hence, we have the following exact sequence
	\[
	H_2(M\backslash \interior{\Gamma(z)}, \partial M \backslash \partial \Gamma(z)) \stackrel{i_*}{\to} H_2(M,\partial M) \stackrel{\varphi}{\to} H_2(\Gamma(z), \partial \Gamma(z))
	\]
	which plays the same role in the case when $\partial M$ is empty.
\end{proof}

\begin{lemma}[{\cite[Lemma 3.2]{Sc2}}]
\label{avoidintersection} 	Suppose $(T,\partial T)$ and $(S,\partial S)$ are properly embedded surfaces in general position in the 3-manifold $(M,\partial M)$, and $T$ is homologous rel $\partial$ to a surface in $M$, which is disjoint from the interior of $S$. Then:
	
	(a) $\Lambda = T\cap S$ has trivial fundamental class in $H_1(S,\partial S)$,
	
	(b) $T \oplus kS$ is properly isotopic to a surface disjoint from $S$ for some sufficiently large positive integer.
\end{lemma}

\begin{lemma}\label{minavoid}Let $z$ be an element in $H_2(M,\partial M;\Z)$ and $\varphi$ be the restriction map $\varphi:H_2(M,\partial M;\Z) \to H_2(\Gamma(z),\partial \Gamma(z);\Z)$. Suppose the image of $u$ under the map $\varphi$ is 0. We can find a properly norm-minimizing surface $\Sigma$ representing $nz+u$ in $H_2(M,\partial M;\Z)$ which avoids the guts of $z$, where $n$ is a sufficiently large positive integer.

\end{lemma}

\begin{proof}
	First, we find an integer $m$ such that $mz +u$ and $z$ lies in the same closed face of Thurston's norm unit ball. Denote $F(z)$ as a facet surface of $z$ and $\Sigma'$ is a properly norm-minimizing surface of $mz+u$ in $M$. By Lemma \ref{avoid}, $\Sigma'$ is homologous to a surface in $M'$ where $M' = M\backslash \interior{\Gamma(z)} $. Hence by Lemma \ref{avoidintersection}, the double curve sum $\Sigma''$ of finitely many copies of $F(z)$ and $\Sigma'$ in the neighborhood of $F(z)$ after an isotopy will be disjoint with $F(z)$. Let $\gamma(z)$ be the sutures of $\Gamma(z)$. Because $\Sigma''$ is disjoint with $F(z)$, $\Sigma''$ intersects $\gamma(z)$ in a union of essential closed curves. Note that $\Sigma''\cap \gamma(z)$ might have components of opposite orientation to the orientation of $\gamma(z)$. However, we can add more copies of $F(z)$ to $\Sigma''$ and for closed curves of different orientation in $\gamma(z)$, we can do a cut-and-paste surgery along the annuli formed by those curves to have a new surface $\Sigma$ such that all closed curves of $\Sigma \cap  \gamma(z)$ are in the same orientation of $\gamma(z)$.

	Then $\varphi([\Sigma])$ is exactly represented by the intersection $\Sigma \cap \Gamma(z)$. We are going to show that $\Sigma \cap \Gamma(z)$ is parallel to some components of $R_+(z)$ and $R_-(z)$. Let $(\Gamma^i(z),R^i_+(z),R^i_-(z),\gamma^i(z))$ be the connected components of $\Gamma(z)$ and $\Sigma^i = \Sigma \cap \Gamma^i(z)$.

By applying the long exact sequence of $(\Gamma^i(z),\partial \Gamma^i(z),\gamma^i(z))$, we have
	$$H_2(\partial \Gamma^i(z),\gamma^i(z)) \to H_2(\Gamma^i(z),\gamma^i(z)) \to H_2(\Gamma^i(z),\partial \Gamma^i(z)). $$
	
	Since $[\Sigma ^i]$ is zero in $H_2(\Gamma(z),\partial \Gamma(z))$, we can find some components $R^i$ (with multiplicities) of $R^i_+(z)$ and $R^i_-(z)$ which is homologous to $\Sigma^i$. Because $\Sigma$ is a properly norm-minimizing surface in $M$, $\Sigma^i$ is taut in $H_2(\Gamma^i(z),\gamma^i(z))$. Since $R^i_+(z)$ and $R^i_-(z)$ are also taut surfaces in $H_2(\Gamma^i(z),\gamma^i(z))$, we have $\chi_-(\Sigma^i) = \chi_-(R^i)$. If we take large enough pieces of  $R_+(z)$ and $R_-(z)$ and replace $R^i$ with $\Sigma^i$, we obtain a taut surface $\Omega^i$ representing a multiple of $[R^i_-]$ in $H_2(\Gamma^i(z),\gamma^i(z))$. By Lemma \ref{lem:reducedhorizontallyprime}, $\Gamma^i(z)$ is horizontally prime. Since $\Omega^i$ consists of horizontal surfaces and $\Sigma^i$ is a union of some components of $\Omega^i$, we know that $\Sigma^i$ is parallel to some components of $R^i_+(z)$ and $R^i_-(z)$ of $\Gamma^i(z)$.

	Therefore, we can isotope $\Sigma^i$ off the guts $\Gamma^i(z)$ for each $i$, i.e., isotope $\Sigma$ off the guts $\Gamma(z)$.
\end{proof}

\begin{lemma} \label{lem:annulus}
	Let $z$ be an element in $H_2(M,\partial M;\Z)$ and $S$ is a properly norm-minimizing surface representing a multiple of $z$. Let $A$ be an essential annulus or an essential disk. If $F(z)$ is a facet surface representing $z$ and containing $S$, we can isotope $A$ such that $A$ is outside of the guts of $F(z)$.
\end{lemma}

\begin{proof}
We isotope $A$ so that $A$ intersects $F(z)$ in minimal number of components. 

When $A$ is a disk, since $F(z)$ is properly norm-minimizing, $F(z)$ intersects $A \cap \partial M$ in the same orientation, and hence $F(z)$ intersects $A$ in the same orientation. So $A\spl F(z)$ is a union of product disks which is outside of the interior of the guts of $F(z)$. Therefore, we can isotope $A$ a little bit to avoid the guts.

When $A$ is an annulus, $A \spl F(z)$ is a union of annuli. Let $A_0$ be a component of $A \spl F(z)$. Suppose $A_0$ is not a product annulus for $M\spl F(z)$ and both boundary curves of $A_0$ lie on the $R_+$ part of $M\spl F(z)$. Then we can do a cut-and-paste along $A_0$ for $R_+$ so that we have a new properly norm-minimizing surface $R'$ after throwing the null-homologous components. Since $F(z)$ is a facet surface, $R'$ is parallel to some components of $R_-$. As a subset of $R'$, $A_0$ can be isotoped to avoid the guts.
\end{proof}

We also need the following lemma: (Cf. \cite[Lemma 0.6]{Ga2} and \cite[Proposition 2.12]{ni2011dehn} )
\begin{lemma} \label{collection} Let $F$ be a compact surface. For any two collections of disjoint loops $C_0$ and $C_1$ in $F$, if $[C_0] = [C_1]$ in $H_1(F,\partial F)$, then there exists a sequence of collections of loops:
	\[
	C_0 = \gamma_0, \gamma_1,\cdots, \gamma_m = C_1
	\]
	and a sequence of embedded surface with no components of disks and compatible with the orientation of $F$
	\[
	W_1, W_1, \cdots,W_{m}
	\]
	such that
	\[
	\overline{\partial W_i \backslash (\partial F)} = \gamma_i \cup (-\gamma_{i-1}).
	\]

\end{lemma}
\begin{proof}
Apply \cite[Lemma 3.10]{Ga1} to $C_0 \cup (-C_1)$.
\end{proof}

Now we can start proving Theorem \ref{T2}.

\begin{proof1}

	We are going to prove the theorem by showing the two inclusions.

	First, we show $\textnormal{ker } \varphi \subset V$, which will be proved by showing that the Thurston norm for a small neighborhood of $z$ in $\textnormal{ker } \varphi$ can be represented by an Euler class. Let $F(z)$ be a facet surface of $z$, $M'$ be $ M\backslash \interior{\Gamma(M,F(z))}$ and $X\times I$ be the window with respect to $F(z)$.

We can think of $M' $ as the result of gluing two ends of $F(z) \times [0,1]$ by two ends of $X \times [0,1]$, i.e., 
$$M'\cong F(z) \times [0,1] \sqcup X \times [0,1] / [(f(x),1)\sim (x,0),(g(x),-1)\sim (x,1)]$$
where $f,g$ are two maps from $X$ to $F(z)$.
	
\begin{figure}[bth]

	\begin{center}
		\epsfig{file = 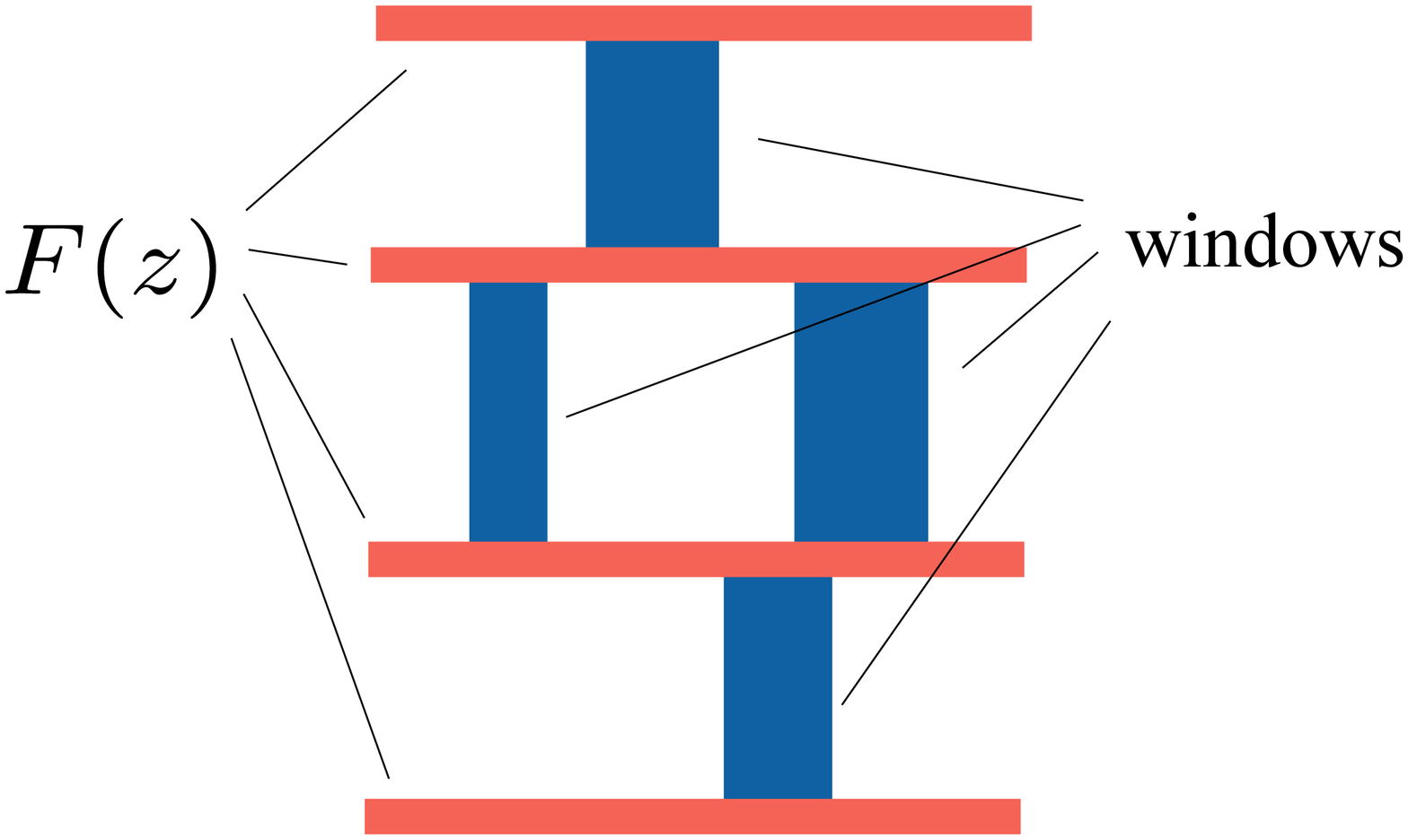, height =.3 \textwidth}
	\end{center}

  \caption{$M' = M\backslash \interior{\Gamma(M,F(z))}$.} 
\end{figure}

Since $u$ is in the kernel of $\varphi$, by lemma \ref{minavoid}, we can find a properly norm-minimizing 
surface $\Sigma$ in $H_2(M,\partial M)$ which represents a multiple of $Nz+u$ such that $\Sigma$ does not intersects $\Gamma(M,F(z))$, i.e., inside $M'$. Besides, $Nz+u$ and $z$ lie in a closed Thurston cone.

Since $M'$ is a union of product sutured manifolds, i.e., $F(z) \times [0,1]$ and windows $X \times I$, it has a natural 1-codimensional fibration $\mathscr{F}$. We fix a Riemannian metric on $M'$ such that the normal vector field $n(\mathscr{F})$ of $\mathscr{F}$ follows the direction of the intervals. This normal vector field will be used in the following proof. 

Here we consider $\partial M =\emptyset$ first. We isotope $\Sigma$ such that $\Sigma \cap X\times [0,1] = C \times [0,1]$, where $C$ is a collection of closed curves because $\Sigma$ avoids the guts $\Gamma(M,F(z))$. In the following we will write $F(z)$ as $F$.

Then $\Sigma_0 \stackrel{\Delta}{=} \Sigma \cap (F\times [0,1]) $ has boundary $ (C_-\times \{0\}) \cup (C_+\times \{1\}) $. If a component $C_i$ of $C_-$ or $C_+$ is trivial in $F$, i.e. there is a disk $D$ in $F$ whose boundary is $C_i$. Then by the incompressibility of $\Sigma$, $C_i$ bounds a disk $D'$ in $\Sigma$. We can isotope $D'$ so that we remove $C_i$ from the intersection of $\Sigma$ and $F\times \{0,1\}$. Hence in the following, we assume no component of $C_-$ and $C_+$ is trivial. We are going to construct a new surface $\Omega_0$ in $F\times [0,1]$ such that $\partial \Omega_0 = \partial \Sigma_0$ and $[\Omega_0 - \Sigma_0] = 0$ in $H_2(F\times [0,1])$.

Since $C_- \times \{0\} \cup C_+ \times \{1\}$ bounds $\Sigma_0$, $[C_- \times \{0\}] - [C_+ \times \{1\}] = 0$ in $H_1(F\times [0,1])$, and hence by projection, $[C_-] = [C_+]$ in $H_1(F)$.

We first isotope $\Sigma_0$ in $F \times I$ such that the intersection of $\Sigma_0$ and $F \times [0,1/2]$ is $C_-\times[0,1/2]$. Then we will do the surgery used in Gabai's paper \cite[Lemma 3.11]{Ga1}. Let W be a subsurface of $F$ such that $\partial W = \gamma' \cup (-\gamma)$ where $\gamma$ is $C_-$ and $\gamma'$ is disjoint with $C_-$. Then the double curve sum $\Sigma_1 = \Sigma_0\oplus F\times (1/4)$ can be isotoped slightly so that $\Sigma_1$ intersect $F\times (1/4)$ at  $\gamma'$. This can be seen by performing the surgery in two steps. First do the surgery along the curves $\gamma$, and isotope the resulting surface $P$ slightly so that $P \cap F\times (1/4) = \gamma'$. Finally do the surgery along the remaining curves to get $\Sigma_1$. If the surgery obtains some spheres or disks, then by the incompressibility of $F$ and $\Sigma_0$, there exist 2 disks $D_1 \in F$ and $D_2 \in \Sigma_0$ such that $D_1\cup D_2$ is a sphere. Because $F$ is the unique norm-minimizing surface in $F\times S^1$ up to isotopy, we know that they bound a 3-cell. Hence we can isotope $\Sigma_0$ vertically to avoid popping out disks or spheres. 

 It now follows from Lemma \ref{collection} that by surgering $\Sigma_0$ and $m$ copies of $F$ we can obtain a new surface (which we still denote as $\Sigma_0$) satisfying the following properties:

1. the normal vector of $\Sigma_0 \cap F\times [0,1/2]$ always has an angle less than or equal to $\pi/2$ with the normal vector $n(\mathscr{F})$ induced from $[0,1]$;

2. $\Sigma_0 \cap F\times [1/2,1]$ will have the same boundary on each sides, i.e. $\Sigma_0 \cap F\times 1/2 =\Sigma_0 \cap F\times 1 = C_+$.

Now we try to replace $\Sigma^* =\Sigma_0 \cap F\times [1/2,1]$ out of $\Sigma$ with some good-behaving surface. Denote $F\times \{1/2\}$ and $F\times \{1\}$ as $F_-$ and $F_+$ respectively. Glue two boundaries of $F \times [1/2,1]$ together and then we get a properly embedded surface $\Sigma'$ in $F \times S^1$. We can stabilize $\Sigma'$ by several pieces of $F_-$ and $F_+$ such that the homology class of $\Sigma'$ in $H_2(F\times S^1)$ is represented by some oriented components of $F$ and vertical tori. Hence we can construct a properly norm-minimizing surface $\Omega'$ by taking the double curve sum of some oriented components of $F$ and vertical tori. Now we decompose $\Omega'$ along the $F_-=F_+$ to obtain a surface $\Omega^*$ having the same boundary as $\Sigma^*$. Hence, $\Omega^* \cup -\Sigma^*$ is a closed surface and $[\Omega^* \cup -\Sigma^*] = 0$ in $H_2(F\times S^1)$, which is therefore 0 in $H_2(F\times [1/2,1],F_-\cup F_+)$. From the exact sequence

\[
H_2(F_-\cup F_+) \to H_2(F\times [1/2,1]) \to H_2(F\times [1/2,1],F_-\cup F_+)
\]
we know that we can stabilize $\Sigma^*$ and $\Omega^*$ with $F_-$ and $ F_+$ such that $[\Omega^* \cup -\Sigma^*]$ is 0 in $H_2(F\times [1/2,1])$. 

We replace $\Sigma^*$ out of $\Sigma$ with $\Omega^*$ to get a new surface $\Omega$. We want to show $\chi_-(\Sigma^*) = \chi_-(\Omega^*)$. If this is true, $\Sigma$ and $\Omega$ have the same homology class and the same Thurston norm. For one direction, since $\Sigma$ is a properly norm-minimizing surface in $M$, $\chi_-(\Sigma^*)$ is smaller than or equal to $\chi_-(\Omega^*)$, otherwise $\chi_-(\Omega) < \chi_-(\Sigma)$. For the other direction, this comes from $\Omega'$ is a properly norm-minimizing surface in $H_2(F\times S^1)$.

In short, $\Omega$ is a properly norm-minimizing surface for some $n'z+u$, where $n'$ is a positive integer. Now we are going to show the Thurston norm of $z$ and $n'z+u$ is represented by a fixed Euler class independent of $u$. 

By the construction of $\Omega$, we know that the angle between $\Omega$ and $\mathscr{F}$ is always less than or equal to $\pi/2$.

	Hence by \cite[Lemma 10.5.7]{CC2}, we know that the 2-plane bundle $T\Omega$ is isomorphic to $\mathscr{F}|\Omega$, which means
	\[
	\chi_-(\Omega) = -e(\mathscr{F})(\Omega)
	\]
	as well as
	\[
	\chi_-(F(z)) = -e(\mathscr{F})(F(z))
	\]

	Thus we know that 
		\begin{eqnarray*}
		x(z) &=& -e(\mathscr{F})(z) \\
		x(n'z+u) &=& -e(\mathscr{F})(n'z+u)
	\end{eqnarray*}
	
	Since $u$ is in arbitrary direction, the Thurston norm of a small neighborhood of $z$ in $\textnormal{ker} \varphi$ can be represented by $-e(\mathscr{F})$, which means they lie in the same open face of the Thurston's norm unit ball. Hence $\textnormal{ker} \varphi \subset V$.

When $\partial M \ne \emptyset$, we consider $H_2(F\times[-1,1], \partial F \times [-1,1])$ and can show the same result.

	Next, suppose there is a neighborhood of $z$ in $\Delta$ which does not contain two elements whose restrictions on a boundary component are in opposite orientations. We want to show $V \subset \textnormal{ker} \varphi$. In order to prove this, for any direction in $V$, we take $u$ small enough such that $z$, $z-u$ and $z+u$ lie in the same Thurston cone and the restrictions of $z-u$ and $z+u$ on each boundary component are in the same orientation. We can multiply $z$ and $u$ by an integer such that $z$, $z-u$ and $z+u$ are in $H_2(M,\partial M;\Z)$ and we still use $z$, $z-u$ and $z+u$ to represent these integral elements.

	We take properly norm-minimizing surfaces $S_{z+u}$ and $S_{z-u}$ for ${z+u}$ and ${z-u}$, respectively from Proposition \ref{prop:maggie}, and construct the double curve sum $S_{z+u}\oplus S_{z-u}$ of $S_{z+u}$ and $S_{z-u}$ which represents $2z$ in $H_2(M,\partial M;\Z)$. We denote the union of the essential annuli and the essential disks which come from the double curve sum as $A$. By Proposition \ref{prop:maggie}, $S_{z+u}\oplus S_{z-u}$ is a properly norm-minimizing surface, and hence we can take a facet surface $F(z)$ of $z$ which contains $S_{z+u}\oplus S_{z-u}$. By Lemma \ref{lem:annulus}, we know that $A$ avoids the guts of $F(z)$. Because $S_{z+u}$ and $S_{z-u}$ lie in the neighborhood of $S_{z+u}\oplus S_{z-u} \cup A$, $S_{z+u}$ and $S_{z-u}$ can be isotoped to avoid the guts of $F(z)$. Thus by lemma \ref{avoid}, the image of $z+u$ and $z-u$ under $\varphi$ is zero, which means $\varphi(u) = 0$ and $\varphi(z) = 0$.

\qed	
\end{proof1}

\begin{example}
	Denote $\Sigma_2$ as a genus 2 closed surface and consider the trivial circle bundle $\Sigma_2 \times S^1$. Then $H_2(\Sigma_2 \times S^1;\R)$ is of rank 5 and the Thurston sphere degenerates to 2 parallel hyperplanes. Then for a non-separating closed curve $c_0$ in $\Sigma_2$, its fiber bundle $c_0\times S^1$ is a $T^2$ which has 0 Thurston norm. Denote $z$ as the homology class of this $T^2$ in $H_2(\Sigma_2 \times S^1;\R)$. There are 2 non-separating closed curves $c_1,c_2$ such that $[c_1\cup c_2] =[c_0]$ in $H_1(\Sigma_2)$ which means $(c_1\cup c_2)\times S^1$ is homologous to $c_0\times S^1$. If we decompose along $c_0\times S^1$ and $(c_1\cup c_2)\times S^1$, we have 2 pairs of pants times $S^1$. Then by decomposing along some product annuli, we will have two solid tori with 4 longitudinal sutures where the longitudes are the $S^1$ fibers. Hence the rank of $H_1(\Gamma(z))\to H_1(\Sigma_2 \times S^1)$ is 1. 
\end{example}

\begin{corollary}[{Cf. \cite[Lemma 4.1]{Ag1}}]
Let $M$ be a compact irreducible orientable 3-manifold with toral boundary and non-degenerate Thurston norm. Let $$z \in H^1(M) \cong H_2(M,\partial M)$$ be in a top dimensional open Thurston cone and there is a neighborhood of $z$ which does not contain two elements whose restrictions on a boundary component are in opposite orientations.

Then we have
$$
H^1(M;\Z) \to H^1(\Gamma(z);\Z)
$$
is 0, i.e.
$$
H_1(\Gamma(z);\Z) \to H_1(M;\Z)/\textnormal{torsion}$$
is 0.	
	
\end{corollary}

\section{Guts and Sutured Decomposition}\label{sec:guts:decomposition}

\begin{namedthm*}{Theorem \ref{T3}}
Let $M$ be an irreducible connected 3-manifold with toral boundary and non-degenerate Thurston norm.	 Let $z$ be an element on a $k$-codimensional $(k>0)$ open Thurston face of the Thurston norm. For any choice of $u$ in any $k'$-codimensional $(k' < k)$ Thurston face whose closure contains $z$, there exists a $w$ on the open segment $(z,u)$ such that $\Gamma(z)$ can be nontrivially decomposed along a properly norm-minimizing surface representing $w$.
\end{namedthm*}

Before proving this theorem, we state a slight generalization of \cite[Theorem 4.1]{FK}.

\begin{definition}
	Given two non-zero cohomology classes $\phi,\psi \in H_2(M,\partial M;\R)$, we say $\phi$ is \emph{subordinate} to $\psi$ if $\phi \in C$ where $C$ is the smallest closed Thurston cone which contains $\psi$.
\end{definition}
\begin{theorem} \label{FK}
Let $M$ be an irreducible connected 3-manifold with toral boundary and let $R$ be a properly norm-minimizing surface. Then for any choice of product decomposition surface (maximal collection of product annuli and product disks) for $M \backslash R\times(-4,4)$ and any choice of $\psi\in H_2(M,\partial M;\Z)$ there exists an $m \in \N$ and a surface $F$ with the following properties:

\begin{enumerate}
	\item $[R]$ is subordinate to $m[R] + \psi$ and $F$ represents $m[R]+\psi$;
	\item $F \oplus(R\times \{-3\}\cup R\times \{3\})$ is a properly norm-minimizing surface;
	\item the intersections $F \cap R \times [-4,-2]$ and $F \cap R \times [2,4]$ are product surfaces;
	\item if $X$ is a gut or a window of $M\backslash R\times(-4,4)$ then $F \cap X$ is a decomposition surface.

\end{enumerate}

\end{theorem}

\begin{proof2}
 By theorem \ref{T2}, $u$ is not in the kernel of the map $\varphi: H^1(M) \to H^1(\Gamma(z))$. Let $F(z)$ be a facet surface for $z$ and fix a product decomposition surface for $M\backslash F(z)\times (-4,4)$.

From Theorem \ref{FK}, we construct a surface $\Sigma$ which represents $mz + u$ such that $\Sigma$ satisfies the 4 properties. So if we decompose $M \spl (F(z) \times \{-3\} \cup F(z) \times \{3\})$ along $\Sigma$, we actually do a sutured decomposition of $M$ along a properly norm-minimizing surface $\Omega \stackrel{\Delta}{=} \Sigma \oplus (F(z) \times \{-3\} \cup F(z) \times \{3\})$ representing some $nz+u$. Since $u$ is not in the kernel of $\varphi$, we have
$$
 [\Sigma \cap \Gamma(z)] \ne 0 \in H_2(\Gamma(z),\partial \Gamma(z);\Z) .
 $$
By Theorem \ref{thm:complexityfunction}, the complexity of $\Gamma(z)\backslash \Sigma$ is smaller than $\Gamma(z)$.

Then $mz+u$ is the $w$ that we want in the statement.
\qed
\end{proof2}

\section{Examples of Guts} \label{sec:guts:example}

In this section, we give some example of guts in some typical manifolds and define the guts of knots.

We show an example of guts as solid tori with 2 sutures and $T\times I$ with 2 sutures.

\begin{example} \label{ex:Book of I-bundles} 	\textbf{Book of I-bundles}. Denote $\Sigma$ as a compact oriented surface and consider $\Sigma \times S^1$. Let $x_1,\cdots,x_k$ be points in $S^1$ and $\gamma_i$ be an essential closed curve in $\Sigma \times \{x_i\}$ for $1\le i \le k$. Let $\eta(\gamma_i)$ be the tubular neighborhood of $\gamma_i$, $m_i$ be the meridian on $\partial \eta(\gamma_i)$ which bounds a disk in $\eta(\gamma_i)$ and $l_i$ be the longitude which is a component of $(\Sigma \times \{x_i\})\cap \partial \eta(\gamma_i) $ such that $\langle m_i,l_i\rangle = 1$. A slope on $\partial \eta(\gamma_i)$ is of the form $\pm(q_i m_i + p_i l_i )$ and is hence determined by the rational number $\frac{q_i}{p_i} \in \Q \cup \{\infty\}. $

	We do a $(q_i/p_i)$-Dehn surgery for each $\gamma_i \times \{x_i\}$ where $|q_i| \ge 2$. If we decompose this manifold along a maximal collection of horizontal surfaces and then decompose along some product annuli that separate $\gamma_i$ from product sutured manifolds, we have $k$ sutured solid tori as guts components.  
	
	Let $m'_i$ be the meridian of the $i$-th sutured solid torus which is represented by $q_i m_i + p_i l_i$ and $l'_i = r_i m_i +s_i l_i$ be a longitude satisfying $\langle m'_i,l'_i\rangle = 1$ which is equivalent to $qs-pr =1$. Since $\langle  m'_i, l_i\rangle  = q_i$ and $\langle  l'_i, l_i\rangle  = r_i$, we have $l_i = -r_i m'_i + q_i l'_i $. Hence the $i$-th solid torus has two sutures with slope of $-r_i/q_i$. 
	
	If we do drillings along $\gamma_i$ instead of Dehn surgery, we will have $k$ guts which is a $T\times I$ with 2 sutures on a boundary component and the other component as a sutured torus.
\end{example}

We also show some examples of guts in links complements.

\begin{figure}[hbt]
\begin{subfigure}[t]{.5\textwidth}
	\begin{center}
  \includegraphics[height =.6 \textwidth]{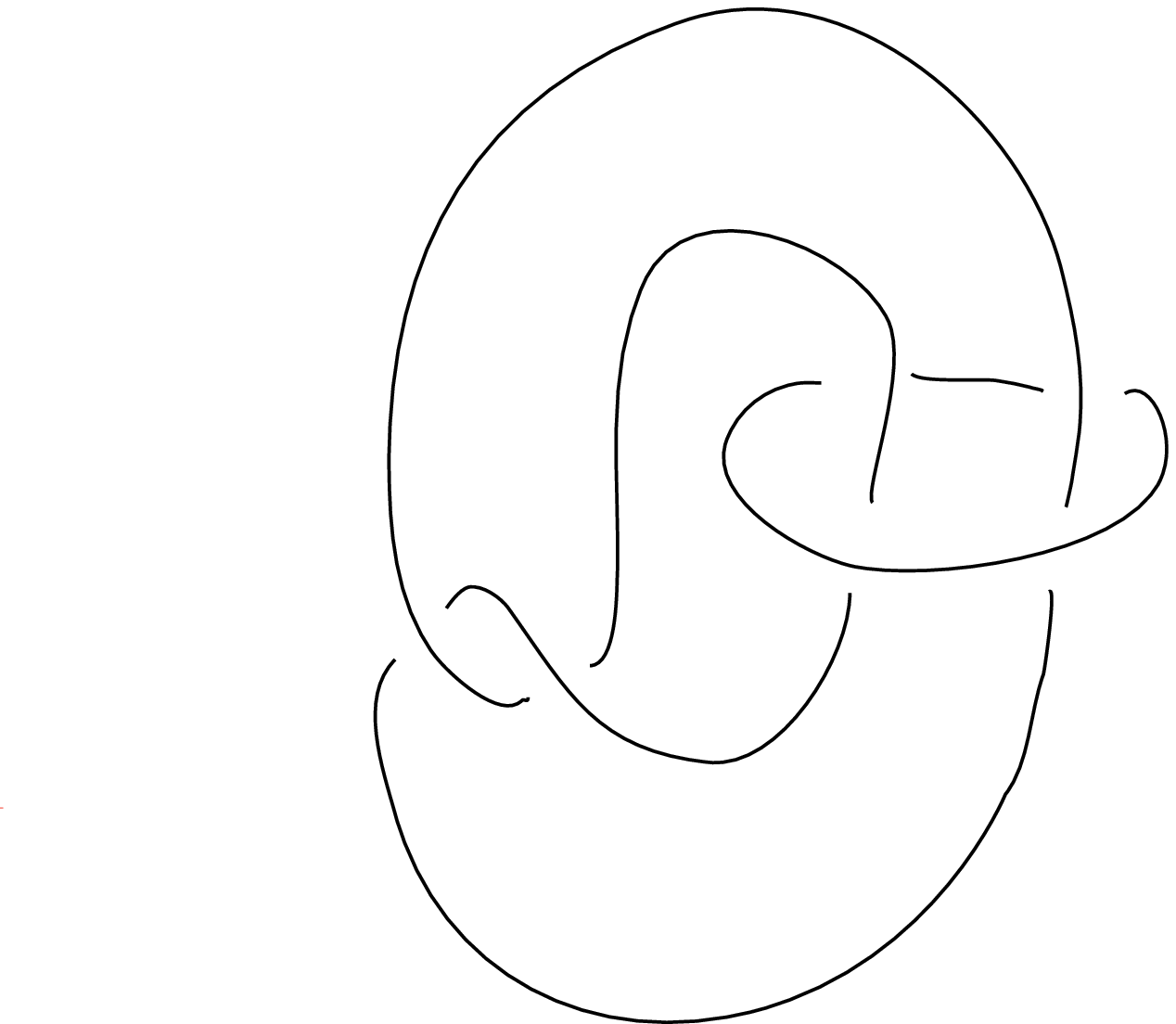}
  \caption{}
\end{center}
  \end{subfigure}%
  \hfill
  \begin{subfigure}[t]{.5\textwidth}
  	\begin{center}
  \includegraphics[height =.6 \textwidth]{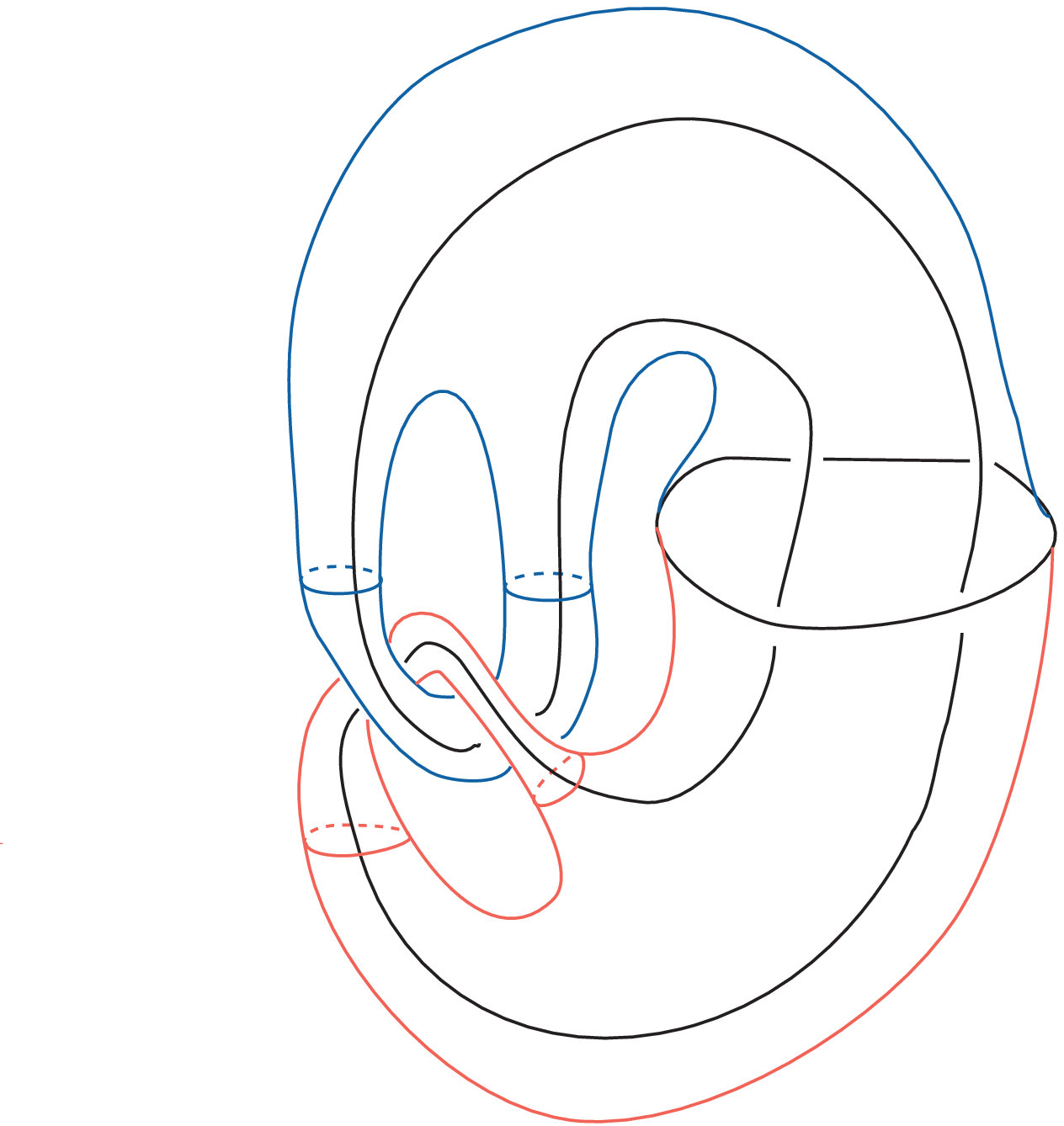}
  \caption{}
\end{center}
  \end{subfigure}%
  \caption{The Whitehead links.} \label{whi}
\end{figure}

\begin{example}
	\textbf{Whitehead links.} We consider the complement $N$ of the Whitehead links. In Figure \ref{whi}, there are 2 Seifert surfaces which are properly norm-minimizing surfaces representing the homology class related to one link component. The blue and red Seifert surfaces are parallel and the blue surface is a facet surface of this link component. Then decomposing along the annuli in orange in the left of Figure \ref{fig:T*I}, we know that the guts of this facet surface is a $T\times I$ with 2 sutures on a boundary component and the other component as a sutured torus.
	
\end{example}

\begin{figure}[hbt]
\begin{subfigure}[t]{.5\textwidth}
	\begin{center}
  \includegraphics[height=.7 \textwidth]{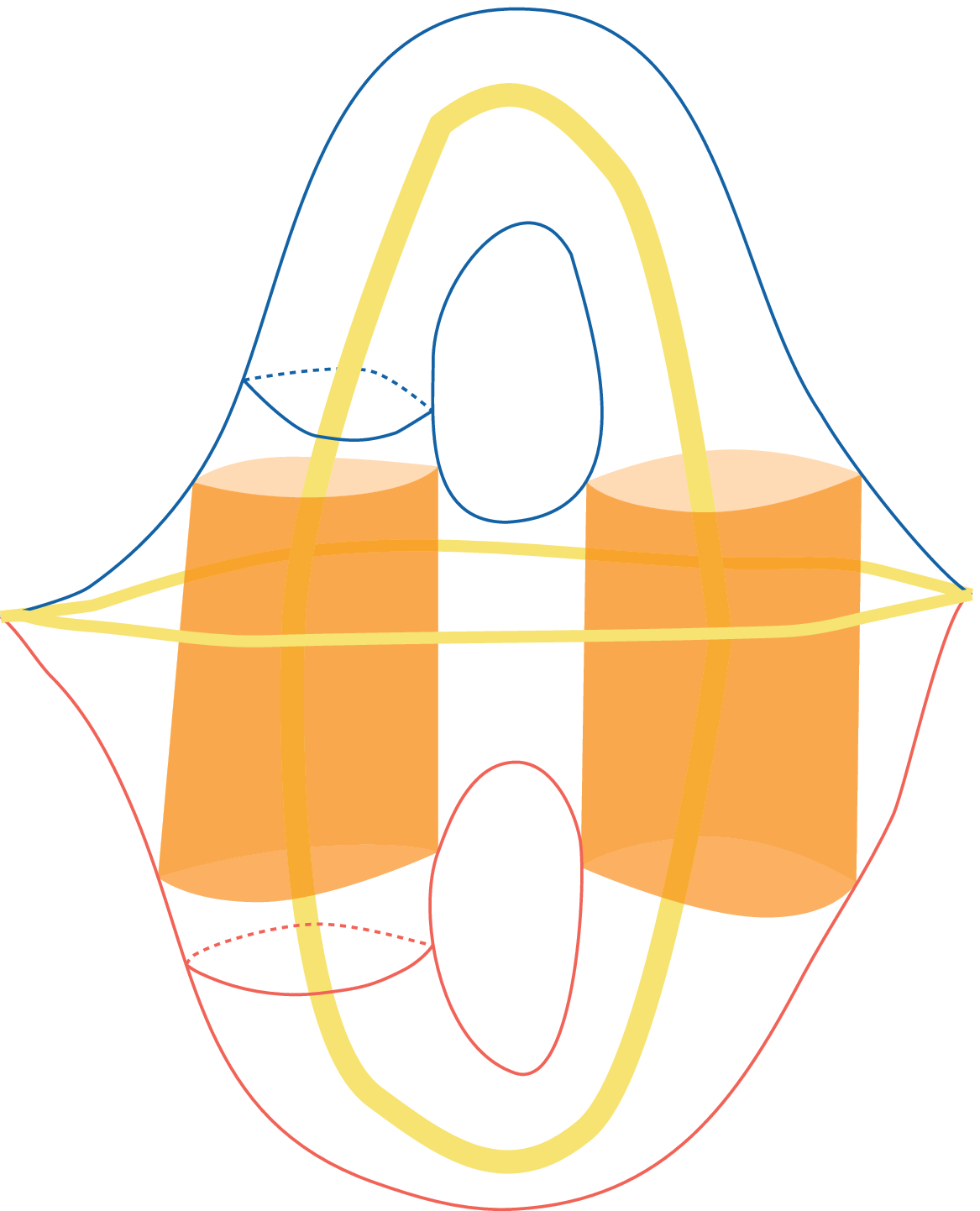}
  \caption{}
\end{center}
  \end{subfigure}%
  \hfill
  \begin{subfigure}[t]{.5\textwidth}
  	\begin{center}
  \includegraphics[height=.7 \textwidth]{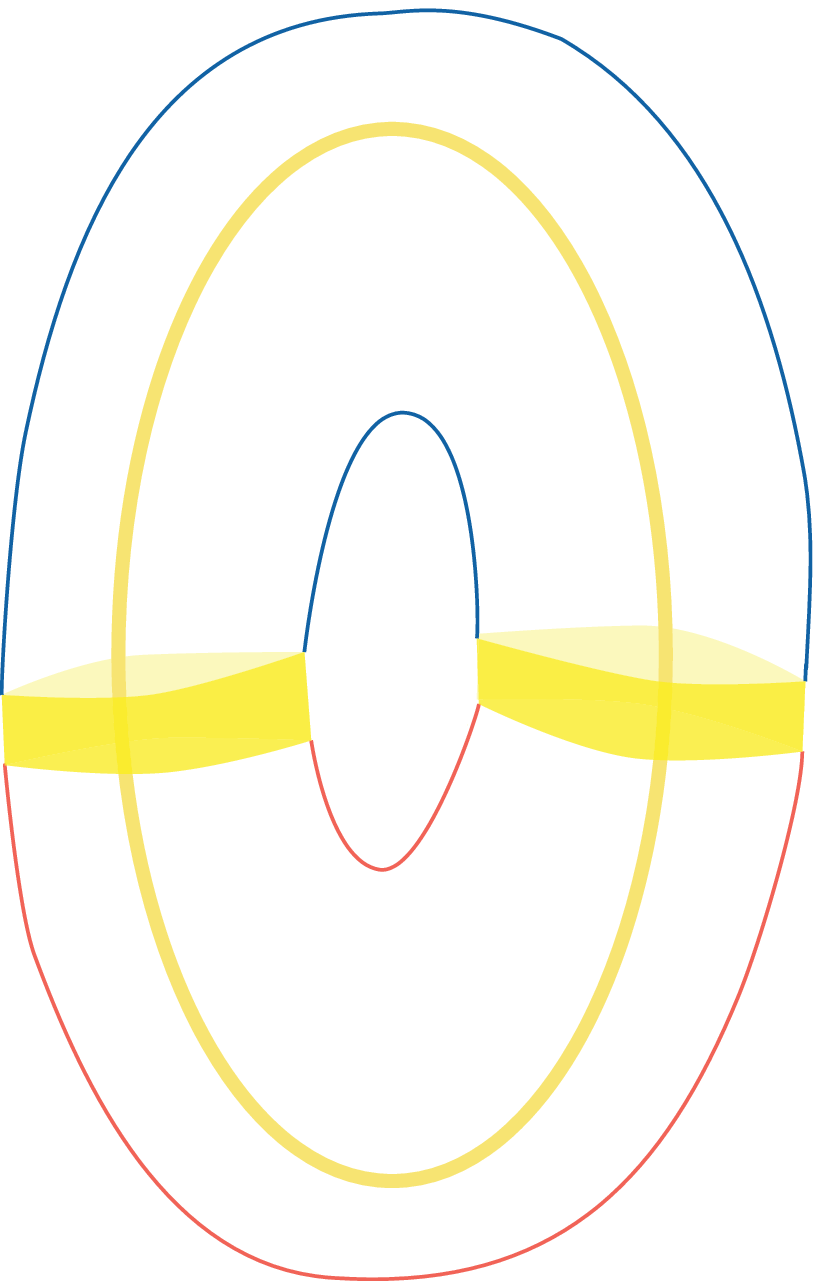}
  \caption{}
\end{center}
  \end{subfigure}%
	\caption{A guts component in the complement of the Whitehead links as well as of the Borromean Rings.} \label{fig:T*I}
\end{figure}

\begin{figure}[hbt]
\begin{subfigure}[t]{.45\textwidth}
	\begin{center}
  \includegraphics[height=.6 \textwidth]{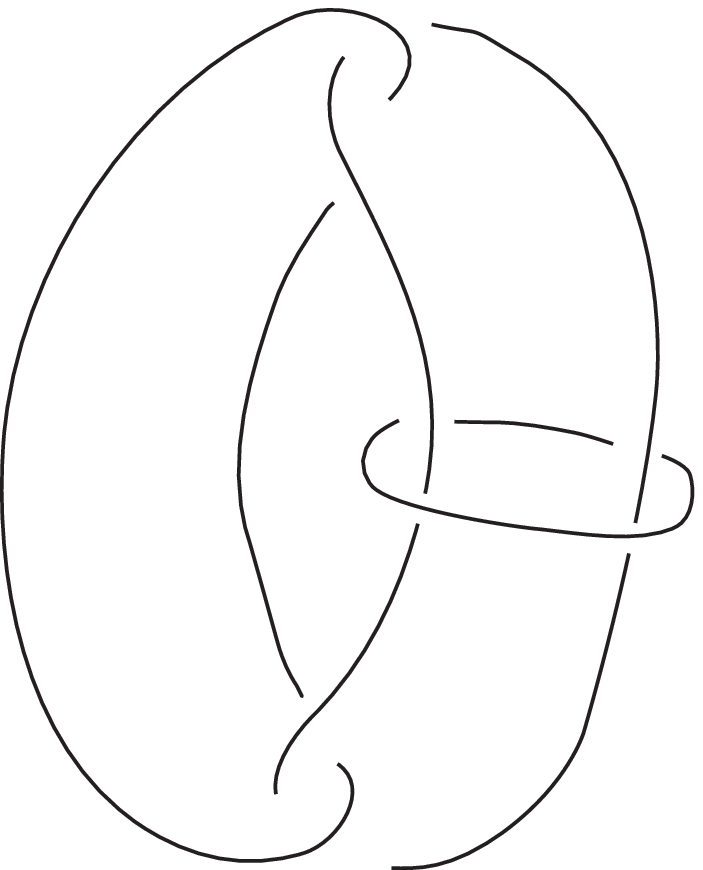}
  \caption{}
\end{center}
  \end{subfigure}%
  \hfill
  \begin{subfigure}[t]{.45\textwidth}
  	\begin{center}
  \includegraphics[height=.6 \textwidth]{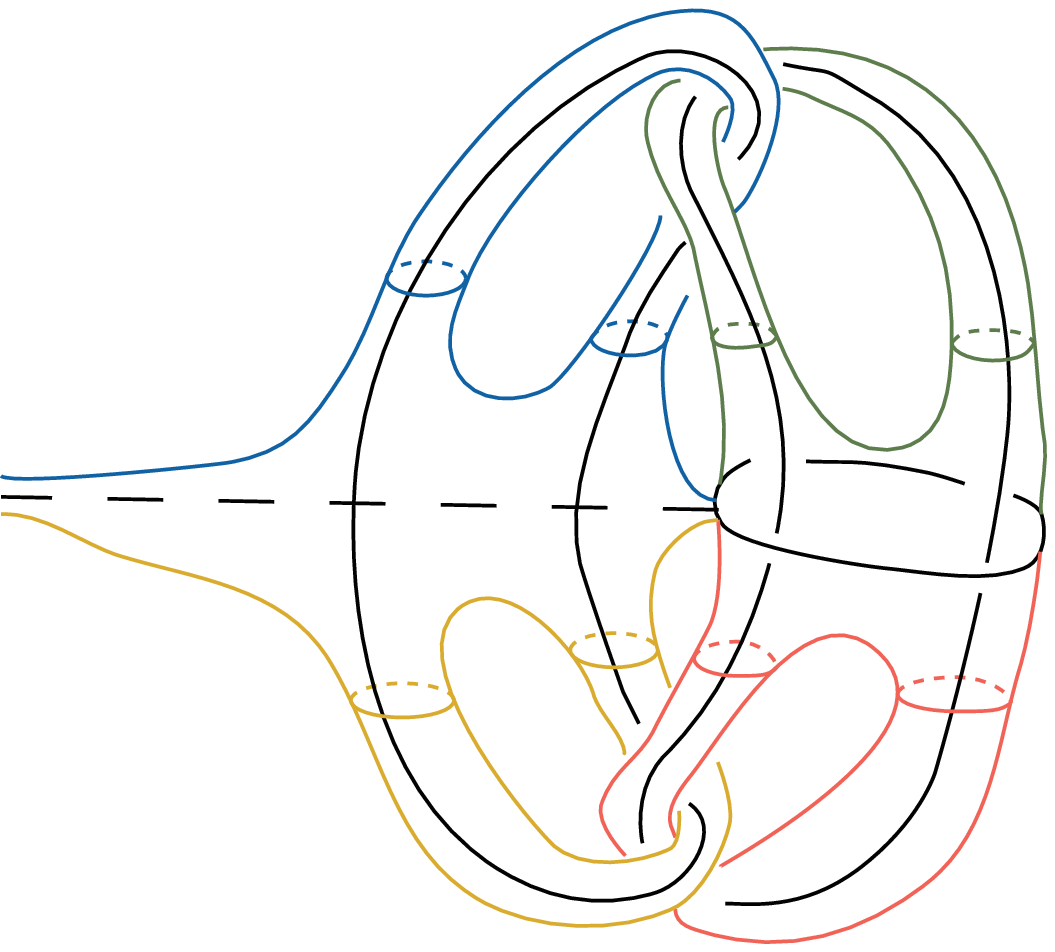}
  \caption{}
\end{center}
  \end{subfigure}%
\caption{The Borromean rings.}\label{bor}
\end{figure}

\begin{example}\label{ex:Borromean} \textbf{Borromean Rings.} We consider the complement $N$ of the Borromean Rings. By \cite[Example 2]{Thurston}, we know that the Thurston ball is an octahedron and its vertices are homology classes that are non-trivial on exactly one link component. In Figure \ref{bor}, there are 4 Seifert surfaces which are properly norm-minimizing surfaces representing the homology class $z$ related to one link component. The blue and green Seifert surfaces are parallel as well as the yellow and red ones. The blue and red surfaces are a facet surface of this link component. Then the guts of this facet surface is two pieces of $T\times I$ with 2 sutures on a boundary component and the other component as a sutured torus as in Figure \ref{fig:T*I}.

A homology class which is non-trivial on exactly two link components is on an edge of the Thurston ball. We can construct a norm-minimizing surface $T$ representing such homology class $w$ by constructing Seifert surface for two link components and cut-and-paste techniques. As stated in Theorem \ref{T3}, we can conduct a non-trivial decomposition of $\Gamma(z)$ along $T$ so that the guts of the resulting sutured manifold is a $T\times I$ with 2 sutures on a boundary component and the other component as a sutured torus.

\end{example}

In the remaining of this section, we analyze the guts of 2-bridge knot complements.

\begin{definition}\label{def:gutsofknot}
	Let $K$ be a knot in $S^3$. The \emph{guts} of $K$ (or the complement of $K$) is the guts of $z \in H_2(S^3\backslash K, \partial (S^3\backslash K))$ where $z$ is the homology class of a Seifert surface. 
\end{definition}

\begin{definition}
	 A \emph{2-bridge knot} is a knot which can be regular isotoped so that the natural height function given by the z-coordinate has only two maxima and two minima as critical points. For a 2-bridge knot, we can assign a rational number $p/q$, with $q$ odd such that it is isotoped to a square "pillowcase" with slope $\pm p/q$. We denote this 2-bridge knot as $K_{p/q}$. See Figure \ref{fig:2bridge}. 
 \end{definition}

\begin{figure}[htb]
\begin{center}
	\includegraphics[width = 4 in]{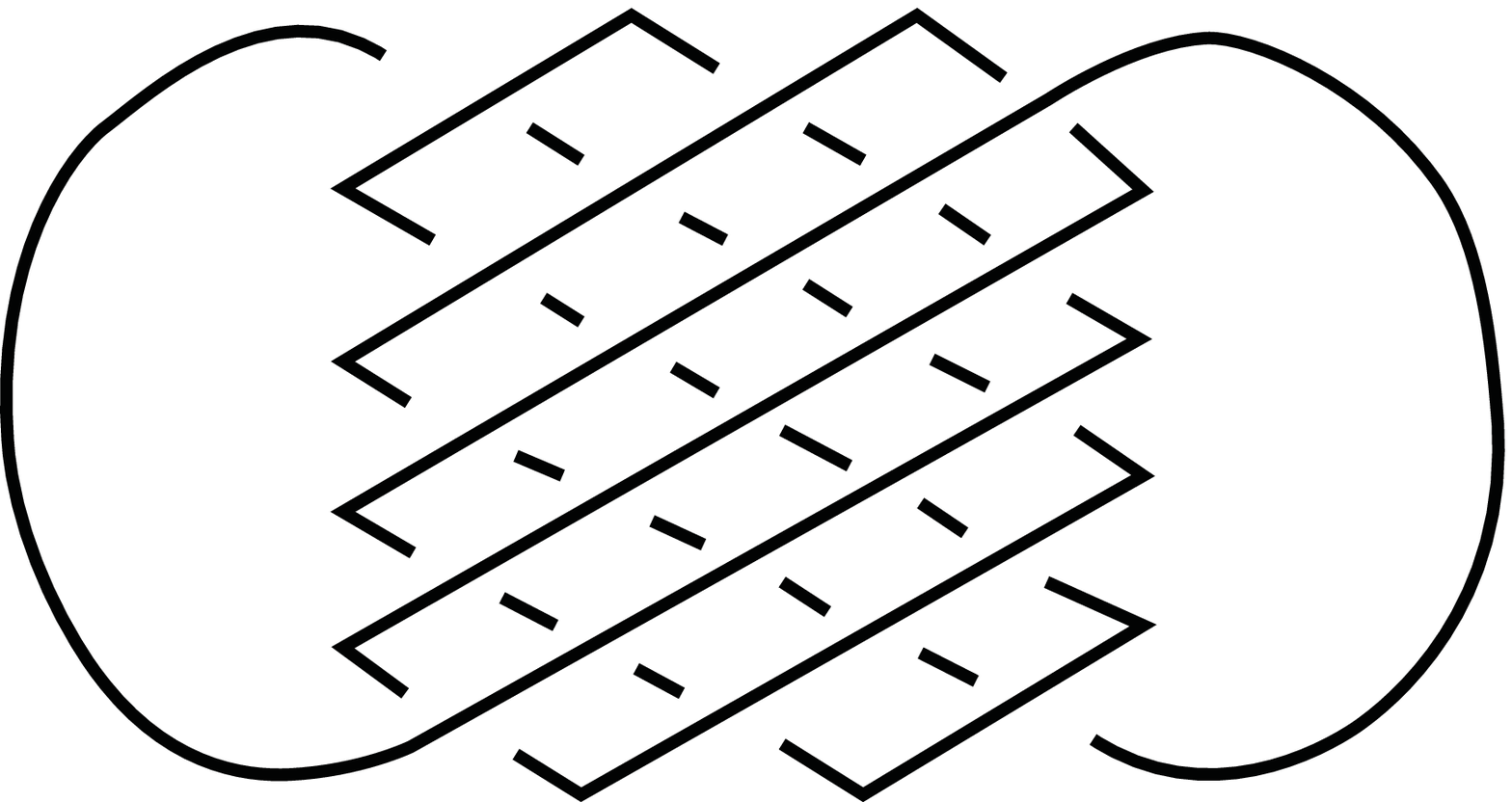}
  \caption{Cf. \cite[Fig. 1]{hatcher1985incompressible}. 2-Bridge Knot $K_{3/5}$.}\label{fig:2bridge}
\end{center}
  
\end{figure}

\begin{theorem}[{\cite[Page 225]{hatcher1985incompressible}}]
	$K_{p/q}$ is isotopic to $K_{p'/q'}$ if and only if $q'= q$ and $p' \equiv p^{\pm 1} (\textnormal{mod} q)$. 
\end{theorem}

\begin{definition}
	
	We define a \emph{continued fraction} as 
	\[
	r + \cfrac{1}{b_1 - \cfrac{1}{b_2 - \cfrac{1}{\cdots - \cfrac{1}{b_k}}}}
	\]
	where $r,b_i \in \Z$, and denote it as $r + [b_1,\ldots,b_k]$. We call $r + [b_1,\ldots,b_k]$ is a \emph{continued fraction expansion} of $p/q$ if $p/q = r + [b_1,\ldots,b_k]$.
\end{definition}

\begin{example}
	\[
	3/5 = 1 + [-2,2]
	\]
\end{example}

Hatcher and Thurston \cite{hatcher1985incompressible} classified all incompressible surfaces, and in particular, oriented Seifert surfaces in 2-bridge knot complements.

\begin{theorem}[{Cf. \cite[Section 1]{conway1970enumeration}, \cite[Page 226]{hatcher1985incompressible}}]
	If $p/q=r+[b_1,\cdots,b_k]$, then $K_{p/q}$ is the boundary of the surface obtained by plumbing together $k$ bands in a row, the $i$-th band having $b_i$ half-twists (right-handed if $b_i > 0$ and left-handed if $b_i < 0$).
\end{theorem}

There are two essentially different ways of performing each plumbing of two adjacent bands. One way of describing this choice is to say that instead of using one of the horizontal plumbing squares, we could use the complement of this square in the horizontal plane containing it, compactified by a point at $\infty $. (Thus we are now regarding $S^3$ as the 2-point compactification of $S^2 \times \R$, with the spheres $S^2  \times  \{*\}$ being horizontal.) See Figure \ref{fig:plumbing}.

\begin{figure}[hbt]
\begin{center}
	 \includegraphics[width = 4 in]{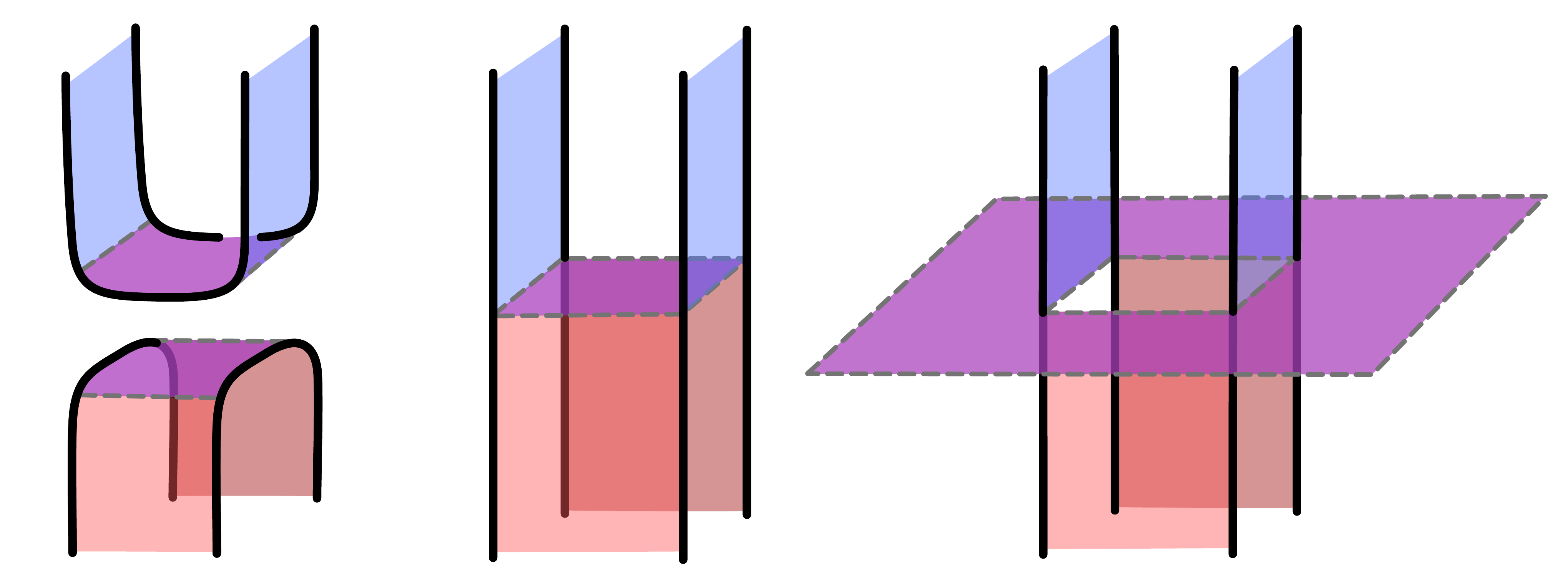}
  \caption{\cite[Figure 6.5]{agol2015certifying} Plumbing.} \label{fig:plumbing}
\end{center}
 
\end{figure}

If we include for each plumbing both of these complementary horizontal plumbing squares, we obtain a certain branched surface $\Sigma[b_1,\ldots, b_k]$. $\Sigma[b_1,\ldots, b_k]$ carries a large number of (not necessarily connected) surfaces, labelled $S_n(n_1,\ldots,n_{k-1})$, where $n\ge 1$ and $0\le n_i \le n$. By definition, $S_n(n_1,\ldots, n_{k-1})$ consists of $n$ parallel sheets running close to the vertical portions of each band of $\Sigma[b_1,\ldots,b_k]$, which bifurcate into $n_i$ parallel copies of the $i$-th inner plumbing square and $n-n_i$ parallel copies of the $i$-th outer plumbing square. For example, when $n= 1$, the surfaces $S_1(n_1, \ldots ,n_{k-1})$ ($n_i=0$ or $1$) are just the $2^{k-1}$ plumbings of the original $k$ bands.

\begin{proposition}[{\cite[Proposition 1]{hatcher1985incompressible}}]
	Orientable incompressible Seifert surfaces for $K_{p/q}$ can be represented by single-sheeted surfaces $S_1(n_1,\cdots,n_{k-1})$ carried by $\Sigma[b_1,\cdots,b_k]$ with each $b_i$ even.
\end{proposition}

\begin{proposition}Let $p/q$ be a rational number with $q$ odd. Then	 there is only one continued fraction expansions $p/q = r+ [b_1,\cdots,b_k]$ for $p/q$ with each $b_i$ even and $|b_i| \ge 2$.
\end{proposition}

\begin{proof}
	Choose an $r$ such that $p/q -r$ has an even numerator. Then we will always have different parities on two sides. For the uniqueness, note that $[b_1,\cdots,b_k]$ has different parities on the numerator and the denominator, and its absolute value is smaller than 1.
\end{proof}

\begin{lemma} Let $p/q = r+ [b_1,\cdots,b_k]$	be the unique continued fraction expansions with each $b_i$ even and $|b_i| \ge 2$. Then the guts of $S^3 - K_{p/q}$ consist of the guts of the complements of $b_i$-half-twisted bands in $S^3$.
\end{lemma}

\begin{proof}
We prove this lemma by induction on $k$. When $k = 1$, the statement holds naturally. 

Suppose the lemma is true for $k-1$. We take a maximal collection of Seifert surfaces for $S^3 - K_{p/q}$ \[S_1(0,0,\cdots,0), S_1(1,0,\cdots,0), S_1(1,1,\cdots,0),\ldots, S_1(1,1,\cdots,1).\] 
	
	At the $({k-1})$-th level, $S_1(0,0,\cdots,0,0), S_1(1,0,\cdots,0,0), \ldots, S_1(1,1,\cdots,1,0)$ have the $i$-th inner plumbing square while $S_1(1,1,\cdots,1,1)$ has the $i$-th outer plumbing square. So locally, it will look like $(k-1)$ pieces of blue surfaces and 1 red surface in Figure \ref{fig:libroid}. Between each two surfaces, we can decompose along 4 product disks coming from the plumbing. So altogether we decompose along the $4(k-1)$ product disks coming from the plumbing. Below the $(k-1)$-th level, there are $(k-1)$ product sutured manifolds and a complement of a $b_{k-1}$-half-twisted band in $S^3$. Above the $(k-1)$-th level, we have a product sutured manifolds coming from $S_1(1,1,\cdots,1,0)$ and $S_1(1,1,\cdots,1,1)$. If we drop this product sutured manifold, the remaining sutured manifolds above the $(k-1)$-th level are equivalent to decomposing $S^3 - K_{p'/q'}$ along $S_1(0,0,\cdots,0), S_1(1,0,\cdots,0), \ldots, S_1(1,1,\cdots,1)$ (only $(k-2)$ coordinates) where $p'/q'=r+ [b_1,\cdots,b_{k-1}]$ and then a product disk between each two surfaces on $(k-1)$-th level. Hence by induction, we know that the lemma holds for any $k$.
	\begin{figure}[hbt]
	\begin{center}
		\includegraphics[width = 3 in]{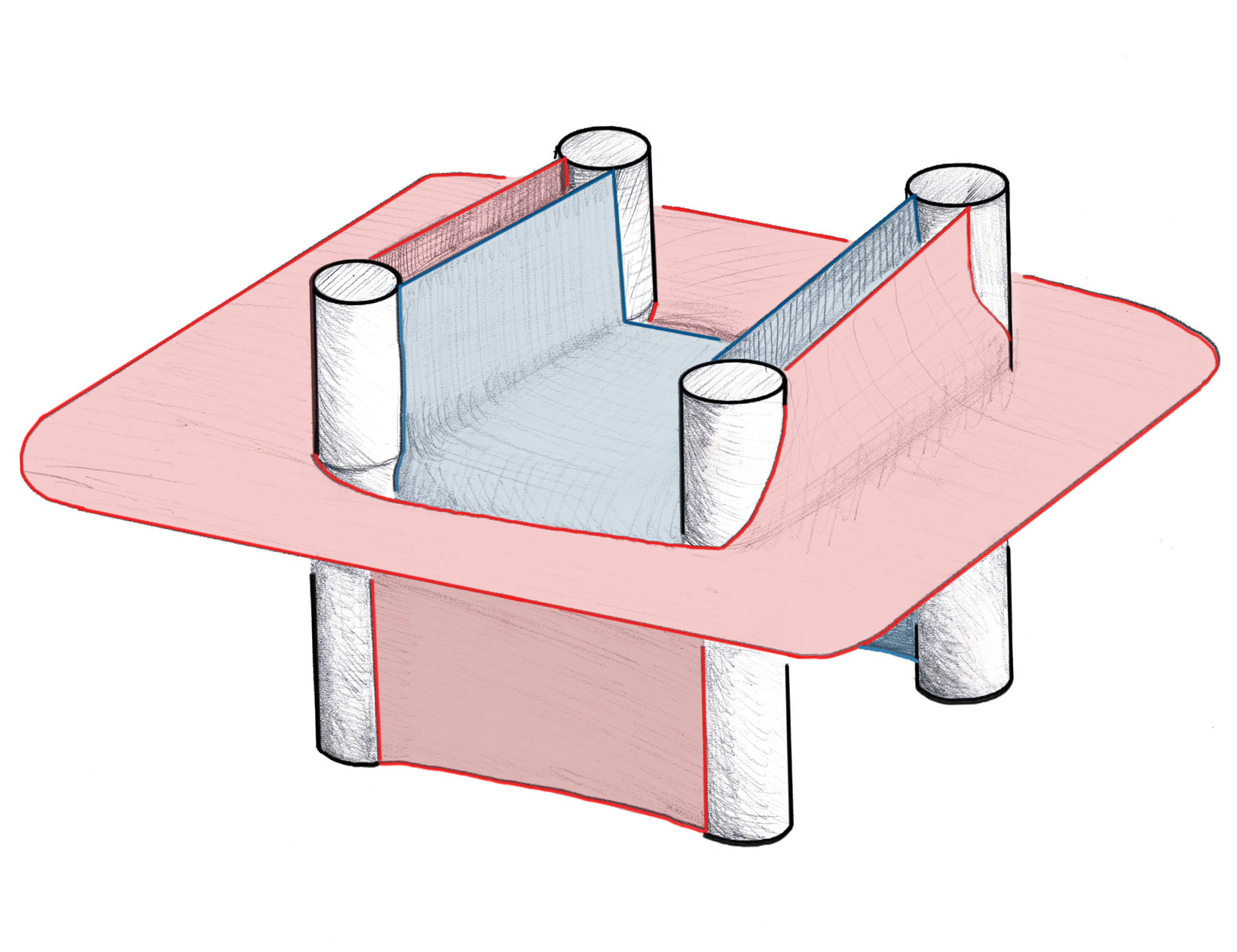}
  \caption{\cite[Figure 6.7]{agol2015certifying} Surfaces bifurcating into inner and outer plumbing squares.} \label{fig:libroid}

	\end{center}
  \end{figure}

\end{proof}

\begin{lemma}
	The complements of a 2-half-twisted band is a product sutured manifold. 
\end{lemma}
\begin{proof}
The boundary of a 2-half-twisted band is the Hopf link. The Hopf fibration is a presentation of $S^3$ as an (oriented) circle bundle over $S^2$ where the fibers of the north and south pole in $S^2$ is the Hopf link. When one deletes the Hopf link, one has a circle bundle over $\R \times S^1$. Projecting once more to the circle factor, one has an open annulus bundle over $S^1$ and a fiber is the interior of a 2-half-twisted band.
\end{proof}

\begin{lemma}
	The complements of a $2k$-half-twisted band is equivalent to a solid torus with two sutures of slope $1/k$.
\end{lemma}
\begin{proof}
	This is because the boundary of a $2k$-half-twisted band is a torus link on a torus $T^2$. We can decompose $S^3$ along this torus into two solid torus. And for the inside solid torus, it is a product sutured manifold.
\end{proof}

\begin{theorem} \label{thm:gutsfor2bridgeknot}
	Let $p/q = r+ [b_1,\cdots,b_k]$	be the unique continued fraction expansions with each $b_i$ even and $|b_i| \ge 2$. Then the guts of $S^3 - K_{p/q}$ consist of sutured solid tori with two sutures of slope $2/b_i$ for each $|b_i| >2$. 
\end{theorem}

We want to show that the guts contain more information than the Knot Floer Homology, i.e. the Alexander polynomial and the signature of two bridge knot.

A knot $K$ is called {\em alternating} if it admits a planar diagram in which the over- and under-passes alternate, as we follow the knot. In the case of alternating knots, knot Floer homology is determined by two classical invariants, the Alexander polynomial $\Delta_K$ and the knot signature $\sigma(K)$. 

\begin{theorem}[Ozsv\'ath-Szab\'o \cite{ozsvath2003heegaard}]
\label{thm:alt}
Let $K \subset S^3$ be an alternating knot with Alexander polynomial $\Delta_K(t) = \sum_{s \in \Z} a_s t^s$ and signature $\sigma = \sigma(K)$.
Then:
$$ \HFKhat_i(K, s) = \begin{cases}
\Z^{|a_s|} & \text{if} \ i=s + \tfrac{\sigma}{2}, \\
0 & \text{otherwise.}
\end{cases}
$$ 
\end{theorem}

\begin{theorem}[{\cite{minkus1982branched}, \cite[Theorem 9.3.6]{murasugi2007knot}}] Suppose $K_{p/q}$ is a 2-bridge knot or link with $0 < p < q, p \text{ odd, and}$ $\gcd(p, q) = 1$. Then
\begin{equation} \label{minkus formula}
\Delta_{K_{p/q}}(t)\doteq \sum_{k=0}^{q-1} (-1)^k t^{\sum_{i=0}^k \epsilon_i}
\end{equation}
where $\epsilon_i = (-1)^{\lfloor i p/q\rfloor}$ and $\lfloor x \rfloor$ is the largest integer less than or equal to $x$.

And the signature, $\sigma(K)$, of K is equal to the number of positive entries
minus the number of negative entries in the sequence of $\{\epsilon_i\} $, which is $\sum_{i=0}^k \epsilon_i$.
\end{theorem}

We find out a pair of two knots with same Alexander polynomial, signature and knot Floer homology but having different guts.

\begin{example} \label{ex:2bridgeknot}
 The Alexander polynomial for $K_{11/15}$ is $\Delta_{K_{11/15}}(t) \doteq 4 - 7t + 4t^2$ which is the same to the one for $K_{7/15}$. Furthermore, they have the same signature $3$. However, $11/15 = 1 + [-4, -4]$ and hence by Theorem \ref{thm:gutsfor2bridgeknot}, the guts for the $K_{11/15}$ complement are two pieces of sutured solid tori with two sutures of slope $-1/2$. While since $7/15 = 1 + [-2, -8]$, the guts for the $K_{7/15}$ complement are a sutured solid torus with two sutures of slope $-1/4$.
\end{example}

This example shows that the guts contain some new information of knots.

\section{Guts and Kakimizu Complex}  \label{sec:guts:kakimizu}

In this section, we apply the properties of guts to the Kakimizu Complex of a homology class in a 3-manifold. 

We use the same notation as \cite{przytycki2012contractibility}. For more detailed description of a Kakimizu complex of a 3-manifold, one may refer to \cite{przytycki2012contractibility}.

Let $E$ be a compact connected orientable, irreducible and $\partial$-irreducible 3-manifold. Let $\gamma$ be a union of oriented disjoint simple closed curves on $\partial E$, which does not separate any component of $\partial E$. We fix a class $\alpha$ in the homology group $H_2(E,\partial E,\Z)$ satisfying $\partial \alpha = [\gamma]$. A \emph{spanning surface} is an oriented surface properly embedded in $E$ in the homology class $\alpha$ whose boundary is homotopic with $\gamma$.

We now define the simplicial complex $MS(E, \gamma, \alpha)$. The vertex set of $MS(E,\gamma,\alpha)$ is defined to be $\mathcal{MS}(E,\gamma,\alpha)$, the set of isotopy classes of spanning surfaces which have minimal genus, i.e. minimize the Thurston norm. We span an edge on $\sigma,\sigma' \in \mathcal{MS}(E,\gamma,\alpha)$ only if they have
representatives $S\in\sigma,S'\in \sigma'$ such that the
(connected) lift of $E\setminus S'$ to the infinite cyclic cover
associated with $\alpha$ intersects exactly two lifts of
$E\setminus S$. In other words, 
this means that the \emph{Kakimizu distance} between $\sigma$ and
$\sigma'$ equals one. Simplices are spanned on all complete subgraphs of the 1–skeleton, i.e. ${MS}(E,\gamma,\alpha)$ is the flag complex spanned on its 1–skeleton.

\begin{definition}
	We call a simplex in the Kakimizu Complex \emph{maximal} if it can not be properly contained in another simplex.
\end{definition}

The goal of this section is to prove the following theorem.
\begin{theorem} \label{thm:kakimizu}
	Let $M$ be an irreducible, orientable $3$-manifold with boundary a disjoint union of tori $\sqcup_{i=1}^n P_i$ and non-degenerate Thurston norm. The dimension of a maximal simplex in a Kakimizu Complex ${MS}(E,\gamma,\alpha)$ is an invariant of the Kakimizu Complex.
\end{theorem}

Let $\Delta$ be a maximal simplex of the Kakimizu Complex ${MS}(E,\gamma,\alpha)$ and denote its vertices as $\sigma_0,\ldots, \sigma_n$. As in \cite[Section 3]{przytycki2012contractibility}, we can choose  representatives $S_0,\ldots, S_n$ of the vertices  such that each two surfaces are almost transverse and have simplified intersection. Therefore the distance between each two surfaces is 1 and $S_0,\ldots, S_n$ are pairwise disjoint. Let $S$ be the union of $S_0,\ldots, S_n$.

Denote $\widetilde{M}$ as the infinite cyclic cover
 of $M$ associated with the (kernel of the) element
of $H^1(M,\Z)$ dual to $\alpha$. Let $\pi \colon \widetilde{M}\rightarrow M$ be the covering map. Let $\tau$ be the generator of the group of covering transformations of $\widetilde{M}$. Pick a base point $x_0$ disjoint with $S$ and let $\tilde x_0$ be a lift of $x_0$. The hypothesis that $\gamma$ does not separate the components of $\partial M$ guarantees that $N_i \stackrel{\Delta}{=} M\setminus S_i$ is connected. Let $\tilde N^0_i$ denote the lift of $N_i$ to $\widetilde{M}$ that contains $\tilde x_0$ and denote $\tilde N^j_i=\tau^j(\tilde N^0_i)$ for $j\in \Z$. Denote also $\tilde S^j_i=\overline{\tilde N^{j-1}_i}\cap \overline{\tilde N^j_i}$ for $j\in \Z$.

For a point $q$ in $\tilde M$ that is disjoint with the lifts of $S_i$, we define the \emph{potential} function $\phi_i(q)$ to be the index $j$ of $\tilde N^j_i$ which contains $q$. We denote $\tilde S$ as the union of lifts of $S$ and let $\Phi(q) = \sum_{i=0}^n \phi_i(q)$ for $q$ in the complements of $\tilde S$.

\begin{lemma} \label{lem:unique}
	If $\Phi(q) = \Phi(q')$, then $\phi_i(q) = \phi_i(q')$ for $0 \le i \le n$.
\end{lemma}
\begin{proof}
	Suppose $\Phi(q) = \Phi(q')$ and the equalities $\phi_i(q) = \phi_i(q')$ do not hold for all $i$. Then there exist $i,i'$ such that
  \[\phi_i(q) < \phi_i(q')\]
  and
  \[\phi_{i'}(q) > \phi_{i'}(q').\]
  
So $\phi_i(q') = \phi_i(q) + k$ and $\phi_{i'}(q') = \phi_{i'}(q) - k'$ where $k,k' \ge 1$. Denote $j$ and $j'$ as $\phi_i(q)$ and $\phi_{i'}(q)$. Then $\tilde N^j_i$ intersects $\tilde N^{j'}_{i'}$ and $\tilde N^{j+k}_i$ intersects $\tilde N^{j'-k'}_{i'}$. By applying the covering transformation $\tau$ on $\tilde N^{j+k}_i$ and $\tilde N^{j'-k'}_{i'}$, we know that $\tilde N^{j}_i$ intersects $\tilde N^{j'-k'-k}_{i'}$. Therefore, $\tilde N^{j}_i$ intersects both $\tilde N^{j'}_{i'}$ and $\tilde N^{j'-k'-k}_{i'}$. This means the distance between $S_i$ and $S_i'$ is at least $k+k' \ge 2$ (see \cite[Definition 2.1]{przytycki2012contractibility}), which is a contradiction.

\end{proof}	

\begin{lemma} \label{lem:monotonic}
	If $\Phi(q') = \Phi(q) + 1$, then there exists an $i_0$ such that $\phi_{i_0}(q') = \phi_{i_0}(q) + 1$ and $\phi_{i}(q') = \phi_{i}(q)$ for $i \ne i_0$.
\end{lemma}
\begin{proof}
 Since $\tilde M$ is connected, there is a path $\eta$ connecting $q$ and $q'$, and furthermore $\eta$ is transverse to the lifts of $S$. We know that when $\eta$ passes a component of lifts of $S$, the value of $\Phi$ is increased or decreased by 1. So there exist a pair of points $p$ and $p'$ on $\eta$ such that $\Phi(p) = \Phi(q)$, $\Phi(p') = \Phi(q')$. Furthermore, $p$ and $p'$ is separated by a $S_{i_0}^{j+0}$. Then $\phi_{i_0}(p') = \phi_{i_0}(p) + 1$ and $\phi_{i}(p') = \phi_{i}(p)$ for $i \ne i_0$. 
 
By Lemma \ref{lem:unique}, we know that $\phi_{i}(q) = \phi_{i}(p)$ and $\phi_{i}(q') = \phi_{i}(p')$ for all $0 \le i \le n$. Hence the equalities in the lemma hold.
\end{proof}

Now we start the proof of Theorem~\ref{thm:kakimizu}.

\begin{proof}[Proof of Theorem~\ref{thm:kakimizu}]
By the preceding lemmas, for each $k \in \Z$, $\Phi(q) = k$ determines unique solutions of $\{\phi_i(q)\}$. Furthermore, increasing $\Phi(q)$ by 1 will result in increasing a $\phi_{i_0}(q)$ by 1. We call this property the \emph{monotonicity} of $\Phi$.

Denote $U(k)$ as $\{q|\Phi(q) = k\}$. Then $\cup_{k \in Z} U(k) = \tilde M \setminus \tilde S$ and $\overline{U(k)}\cap \overline{U(k')} \ne \emptyset$ if and only if $|k-k'| = 1$. Furthermore for an integer $k_0$, $\overline{U(k_0)}\cap \overline{U(k_0+1)}$ is a subset of $\tilde S_{i_0}^{j_0}$ by Lemma \ref{lem:monotonic}. Note that $\tilde S_{i_0}^{j_0}$ is $\overline{\tilde N^{j_0-1}_{i_0}}\cap \overline{\tilde N^{j_0}_{i_0}}$. From the monotonicity of $\Phi$, we know that only one $k$ satisfies that $\overline{U(k)}\cap \overline{U(k+1)}$ is a subset of $\tilde S_{i_0}^{j_0}$. Hence $\overline{U(k_0)}\cap \overline{U(k_0+1)}$ is $\tilde S_{i_0}^{j_0}$.

By arranging the order of $S_i$, we can let $\overline{U(0)}\cap \overline{U(1)}$ be $\tilde S_{0}^{1}$, $\overline{U(1)}\cap \overline{U(2)}$ be $\tilde S_{1}^{1}$, ..., $\overline{U(n)}\cap \overline{U(n+1)}$ be $\tilde S_{n}^{1}$. By applying the projection map $\pi$ on $\overline{U(i)}$ for $0\le i \le n$, we denote $\sL_i = \pi(\overline{U(i)})$. Then $\sL_i \cap \sL_{i+1} = S_i$ (mod $(n+1)$) for $0\le i \le n$.

We say $\sL_i$ is the $i$-th \emph{layer} and think of it as a sutured manifold with $R_- = S_i$ and $R_+ = S_{i+1}$. We call the combination of layers a \emph{sutured bundle} structure for $M$ with respect to $S$. 

Since $S_i$ is not isotopic to $S_{i'}$ for $i \ne i'$, each $\sL_i$ is not a product sutured manifold. 

Suppose there is a properly norm-minimizing surface $T$ such that $[T] = \alpha$, $T$ is disjoint with $S$ and a component of $T$ is not parallel to $S$. Without loss of generality, we assume that component is in $\sL_0$. Let $T'$ denote $T \cap \sL_0$. Since $T$ is homologous to $S_0$, there is a region $N$ of $M$ such $\partial N$ is $T^- \cup S_0$ where $T^-$ comes from reversing the orientation of $T$. Then the boundary of $N\cap \sL_0$ is $(T')^- \cup S'_0$ where $S'_0$ is the union of some components of $S_0$. Hence $T'$ is homologous to $S'_0$. Denote $T_0$ as $(S_0 - (S'_0)) \cup T'$ and perturb $T_0$ in the neighborhood of $S_0$ so that $T_0$ is in the interior of $N$. Since $T'$ is not parallel to $S$, we know that $T_0$ represents a vertex of ${MS}(E,\gamma,\alpha)$ different from $S_0, \ldots, S_n$. By adding $T_0$, we have a sutured bundle with $(n+2)$ layers and hence the distance between $T_0$ and $S_i$ is 1 for $0 \le i \le n$. Therefore $T_0, S_0, \ldots, S_n$ forms a simplex contains the maximal simplex $\Delta$ which is a contradiction.

Let $F$ be a facet surface contains $S$. From the above, we know that $M\spl F$ is the union of $M\spl S$ and some product sutured manifolds. Therefore the guts $\Gamma(M,F)$ of $M\spl F$ is the guts of $M\spl S$. By Lemma \ref{lem:atmostonegut}, the guts of $\sL_i$ is connected, i.e. a component of the guts of $\Gamma(M,F)$. Hence the number of components of $\Gamma(M,F)$ is $n+1$. In other words, the dimension $n$ of the maximal simplex $\Delta$ is $|\Gamma(M,F)|-1$. By Theorem \ref{thm:gutsforhomology} and Definition \ref{def:gutsforhomology}, $n$ is equal to $|\Gamma(\alpha)| -1$, which does not depend on the selection of the maximal simplex $\Delta$. 
	
\end{proof}

\nocite{juhasz2008knot}

\section{Conclusion}

The main result in this paper was motivated by understanding the relation between the Thurston norm and sutured manifolds. In a sequel, this will be applied to understand the minimal volume 3-cusped hyperbolic manifolds. As described before, it will be interesting to see the ramifications for the structure of the Kakimizu complex. 

The main theorem in this paper should generalize to give guts for homology classes of sutured manifolds. There is a polytope associated to the homology of a sutured manifold, and we expect that these guts will have similar relations with this polytope. 

Another possible direction that the main results of this paper may be useful for is to understand the structure of the Thurston norm polytope. 
 In particular, it is unknown for betti number $>2$ which polytopes can be the unit ball of the Thurston norm of a 3-manifold? The local structure of the polytope near a vertex should be governed by the polytope of the sutured manifold of the face associated to that point. But it is harder to imagine how this local structure could be integrated to get a global picture of the polytope.

The Example \ref{ex:2bridgeknot} indicates that the guts contains more information than sutured Floer homology.  
One may ask whether one can incorporate more information to refine sutured Floer homology? 
In particular, are there homology operations on Floer homology that might distinguish sutured solid tori with one or more components?

\pagestyle{headings}

\bibliographystyle{amsalpha}
\bibliography{guts}

\end{document}